\documentclass[11pt,reqno]{amsart}

\usepackage[T1]{fontenc}
\usepackage[utf8]{inputenc}

\usepackage{newtxtext}
\usepackage{newtxmath}

\usepackage[
a4paper,
left=27mm,
right=27mm,
top=25mm,
bottom=25mm
]{geometry}

\usepackage{mathtools}   
\usepackage{mathrsfs}
\usepackage{extarrows}

\usepackage[shortlabels]{enumitem}
\usepackage{graphicx}
\usepackage{float}
\usepackage{tikz}
\usetikzlibrary{arrows.meta,patterns,decorations.pathreplacing,calc}
\usepackage{booktabs}
\usepackage{xcolor}
\usepackage{microtype}

\setlist{
	topsep=4pt,
	itemsep=1pt,
	parsep=0pt,
	partopsep=0pt
}

\usepackage{xcolor}

\definecolor{myred}{RGB}{220,0,0}
\definecolor{myblue}{RGB}{30,100,230}
\definecolor{myurlblue}{RGB}{40,110,230}

\usepackage[
colorlinks=true,
linkcolor=myblue,
citecolor=myred,
urlcolor=myurlblue
]{hyperref}
\makeatletter
\renewcommand{\tagform@}[1]{%
	\maketag@@@{\textcolor{myblue}{(\ignorespaces#1\unskip\@@italiccorr)}}%
}
\makeatother
\allowdisplaybreaks[2]

\newtheorem{theorem}{Theorem}[section]
\newtheorem{lemma}[theorem]{Lemma}

\newtheorem{theoremA}{Theorem}

\theoremstyle{definition}

\theoremstyle{remark}
\newtheorem{remark}[theorem]{Remark}

\numberwithin{equation}{section}

\newcommand{\nin}{\notin}

\author{Yanping Chen$^{*}$}
\thanks{Corresponding author: Yanping Chen. Email: yanpingch@126.com.}
\address{Department of Mathematics, Northeastern University; State Key Laboratory of Synthetical Automation for Process Industries, Shenyang 110004, China}
\email{yanpingch@126.com}

\author{Yongrui Tang}
\address{Department of Mathematics, Northeastern University, Shenyang 110004, China}
\email{yongruit@163.com}

\subjclass[2020]{ 35Q41, 42B25,42B37}
\keywords{Schrödinger equation, variation operator, maximal estimate,fractal estimate}
\thanks{There search was partially supported by the National Natural Science Foundation
	of China (Grant No.12525105 and 12371092).}

\begin{document}

	\title[]{Sharp Variational estimates for Schr\"odinger equation in $\mathbb{R}^{2}$}


	\subjclass[2000]{35Q41; 42B20; 42B37}
	\date{}

	\dedicatory{}

\keywords{Schrödinger equation; variation operators; maximal estimates; fractal estimates}

	\begin{abstract}
	For $p\geq2$ and $r\geq 1$, we prove the sharp strong $r$-variation estimate for  solutions to the free Schr\"odinger equation
	$$
	\bigl\|V^r_{[0,1]}\{e^{it\Delta}f\}\bigr\|_{L^p(\mathbb{R}^2)}
	\lesssim
	\|f\|_{H^s(\mathbb{R}^2)}.
	$$
	The estimate holds if $s> s(p,r)$ and fails if $s< s(p,r)$, where
$$s(p,r)=\max\left\{
\frac{2}{r}+\frac{1}{p}-\frac12,
\frac{2}{r}-\frac{2}{p},
\frac{1}{p}+\frac{1}{r},
\frac{2}{p}-\frac13,
1-\frac{2}{p}
\right\}.$$
	In particular,
	$s(3,\infty)=1/3$. Thus our result yields a sharp variational extension of the two-dimensional Schrödinger maximal estimate of Du, Guth, and Li [{Ann. Math.}, 186, 607--640 (2017)]. The proof employs fractal restriction estimates to eliminate the regularity loss originating from the summation in the variation norm.
	\end{abstract}

	\maketitle

	\tableofcontents \pagenumbering{arabic}



\section{Introduction}
The solution to the free Schr\"odinger equation
$$
\left\{\begin{array}{ll}
i\partial_tu+\Delta u=0,&(x,t)\in\mathbb{R}^2\times\mathbb{R},\\
u(x,0)=f(x),&x\in\mathbb{R}^2,
\end{array}\right.
$$
is given by
$$
u(x,t)=e^{it\Delta}f(x)
=\frac{1}{(2\pi)^2}\int_{\mathbb{R}^2}
e^{i(x\cdot\xi-t|\xi|^2)}\mathcal{F}_x(f)(\xi)d\xi.
$$

Let ${\mathcal{F}}=\{F(x,t) : t \in I \subset \mathbb{R}\}$ be a family of Lebesgue measurable functions defined on $\mathbb{R}^2$.
For $t\in I:=[a,b]$, $x\in \mathbb{R}^{2}$ and
$1\leq r\leq\infty$, the value of the $r$-variation function $V_I^r\{F(x,t)\}_t$ of the family $\mathcal{F}$ at $x$ is
defined by
$$
V_I^r\{F(x,t)\}_t:=
\left\{\begin{array}{ll}
	\sup_{a\leq t_0<\cdots<t_J\leq b}
	\left(\sum_{j=1}^J
	|F(x,t_j)-F(x,t_{j-1})|^r\right)^{1/r},\quad r<\infty;\\[7pt]
	2\sup_{t\in I}|F(x,t)|,\quad \quad\quad \quad\quad \quad\quad \quad\quad \quad \quad \quad\quad \quad \quad r=\infty.
\end{array}\right.
$$ 
More precisely, for any $t_0\in I$,
$$\sup_{t\in I}|F(x,t)|\leq |F(x,t_0)|+V_I^r\{F(x,t)\}.$$

Suppose $\mathcal{A}=\{A_t\}_{t\in I}$ is a family of operators on $L^p(\mathbb{R}^2)$ with $1 \le p \le \infty$. The associated strong $r$-variation operator $V_I^r\mathcal{A}$ is defined by
$$
V_I^r \{\mathcal{A}(f)(x)\} = V_I^r(\{A_t f(x)\}_{t\in I}).
$$

There are two elementary yet important observations that motivate the study of variational inequalities in ergodic theory and harmonic analysis. First, if $V_I^r(\mathcal{F})(x)<\infty$ for some finite $r$, then $F(x,t)$ converges as $t\to t_0$ whenever $t_0$ is an adherent point of $I$. From this fact, one readily deduces that $A_t(f)$ converges almost everywhere for all $f\in L^p$, provided that $V_I^r\mathcal{A}$ (with finite $r$) is of weak type $(p,p)$ and $p<\infty$. Second, the $r$-variation operator pointwise dominates the maximal function: for any $r\ge 1$,
\begin{eqnarray*}
	\sup_{t\in I}|A_t(f)(x)|\le |A_{t_0}f(x)| + V_I^r \{\mathcal{A}f(x)\}, 
\end{eqnarray*}
where $t_0\in I$ is  any fixed point.

In the present paper, we investigate the sharp boundedness of the $r$-variational Schr\"odinger operator over the local unit time interval and the global space.  Our focus here is the optimal Sobolev regularity required to control the time variation, along with its dependence on the spatial exponent $p$ and the variation exponent $r$.

For
$2\leq p\leq\infty$, define
$$m(p):=
\left\{\begin{array}{ll}
\displaystyle \frac{2}{p}-\frac{1}{3}, &2\leq p\leq3,\\[5pt] 
\displaystyle 1-\frac{2}{p}, &3\leq p\leq\infty,
\end{array}\right.$$
and set
$$s(p,r):=
\left\{\begin{array}{ll}
\displaystyle \frac{2}{r}+\frac{1}{p}-\frac{1}{2}, &1\leq r\leq2,\quad 2\leq p\leq6,\\[7pt]
\displaystyle \frac{2}{r}-\frac{2}{p},&1\leq r\leq2,\quad 6\leq p\leq\infty,\\[7pt]
\displaystyle \max\left\{m(p),\frac1p+\frac1r\right\},&2<r\leq\infty.
\end{array}\right.$$

\begin{theorem}\label{Th-1}
Let $2\leq p\leq\infty$ and $1\leq r\leq\infty$.  If $s> s(p,r)$, then
\begin{eqnarray}\label{adeq-2}
\left\|V^r_{[0,1]}\{e^{it\Delta}f\}\right\|_{L^p(\mathbb{R}^2)}\lesssim_{p,r,s}\|f\|_{H^s(\mathbb{R}^2)}.
\end{eqnarray}
The estimate fails whenever  $s< s(p,r)$. 
\end{theorem}

Figure \ref{fig:admissible-region} shows the five expressions that constitute the boundary of Theorem \ref{Th-1}.
\begin{figure}[H]
	\centering
	\resizebox{0.7\textwidth}{!}{%
		\begin{tikzpicture}[
			x=1.36cm,y=0.65cm,
			axis/.style={line width=0.85pt,-{Latex[length=2.2mm]}},
			boundary/.style={line width=1.05pt},
			guide/.style={gray!70,dashed,line width=0.55pt},
			tick/.style={line width=0.65pt},
			region/.style={align=center,font=\small},
			note/.style={align=center,font=\footnotesize},
			curve label/.style={note,inner sep=1.2pt},
			label arrow/.style={gray!75,line width=0.5pt,-{Latex[length=1.7mm]}}
			]
			\def\xzero{0.35}\def\xtwo{2.10}\def\xthree{5.20}\def\xsix{8.25}\def\xmax{10.65}
			\def\yone{1.45}\def\ytwo{3.15}\def\ysix{7.10}\def\ytop{14.20}
			\fill[gray!10] (\xzero,0) rectangle (\xtwo,\ytop);
			\fill[pattern=north east lines,pattern color=gray!38] (\xzero,0) rectangle (\xtwo,\ytop);
			\fill[gray!10] (\xtwo,0) rectangle (\xmax,\yone);
			\fill[pattern=north east lines,pattern color=gray!38] (\xtwo,0) rectangle (\xmax,\yone);
			\draw[boundary] (\xtwo,0)--(\xtwo,\ytop);
			\draw[boundary] (\xtwo,\yone)--(\xmax,\yone);
			\draw[boundary] (\xtwo,\ytwo)--(\xmax,\ytwo);
			\draw[guide] (\xthree,\yone)--(\xthree,\ytop+0.10);
			\draw[guide] (\xsix,\yone)--(\xsix,\ytwo);
			\draw[guide] (\xzero,\ysix)--(\xtwo,\ysix);
			\draw[boundary,-{Latex[length=2.0mm]}] (\xtwo,\ysix)..controls (2.95,7.55) and (4.72,10.55)..(5.08,13.95);
			\draw[boundary,-{Latex[length=2.0mm]}] (\xsix,\ytwo)..controls (7.05,4.15) and (5.58,9.10)..(5.32,13.95);
			\draw[axis] (\xzero,0)--(\xmax+0.35,0) node[right] {$p$};
			\draw[axis] (\xzero,0)--(\xzero,\ytop+0.75) node[above] {$r$};
			\foreach \x/\lab in {\xtwo/2,\xthree/3,\xsix/6}{\draw[tick] (\x,-0.12)--(\x,0.12);\node[below=2pt] at (\x,0) {$\lab$};}
			\draw[tick] (\xzero-0.10,\yone)--(\xzero+0.10,\yone);\node[left=3pt] at (\xzero,\yone) {$1$};
			\draw[tick] (\xzero-0.10,\ytwo)--(\xzero+0.10,\ytwo);\node[left=3pt] at (\xzero,\ytwo) {$2$};
			\draw[tick] (\xzero-0.10,\ysix)--(\xzero+0.10,\ysix);\node[left=3pt] at (\xzero,\ysix) {$6$};
			\node[left=3pt] at (\xzero,\ytop+0.28) {$+\infty$};
			\node[note,fill=white,inner sep=2.5pt] at (\xthree,15.05) {$(p,r)=(3,+\infty)$\\[-1pt]classical maximal case: $ s=\frac13$};
			\node[region] at (3.20,10.25) {$\displaystyle  s=\frac{2}{p}-\frac13$};
			\node[region] at (5.05,5.00) {$\displaystyle  s=\frac1p+\frac1r$};
			\node[region] at (8.15,8.10) {$\displaystyle  s=1-\frac{2}{p}$};
			\node[region] at (5.05,2.27) {$\displaystyle  s=\frac{2}{r}+\frac1p-\frac12$};
			\node[region] at (9.45,2.27) {$\displaystyle  s=\frac{2}{r}-\frac{2}{p}$};
			\node[curve label] (leftcurve) at (4.20,7.95) {$\displaystyle r=\frac{3p}{3-p}$};
			\draw[label arrow] (leftcurve.west)--(2.90,7.72);
			\node[curve label] (rightcurve) at (7.55,6.05) {$\displaystyle r=\frac{p}{p-3}$};
			\draw[label arrow] (rightcurve.west)--(6.62,6.15);
			\fill (\xtwo,\ysix) circle (1.35pt);\fill (\xsix,\ytwo) circle (1.35pt);
			\node[note] at (1.22,4.10) {excluded range\\$p<2$};
			\node[note] at (6.65,0.68) {excluded range $r<1$};
			\node[font=\scriptsize,gray!65,anchor=north east] at (\xmax,-0.56) {schematic axes};
		\end{tikzpicture}%
	}
	\caption{The critical Sobolev index in the admissible $(p,r)$-region.}
	\label{fig:admissible-region}
\end{figure}
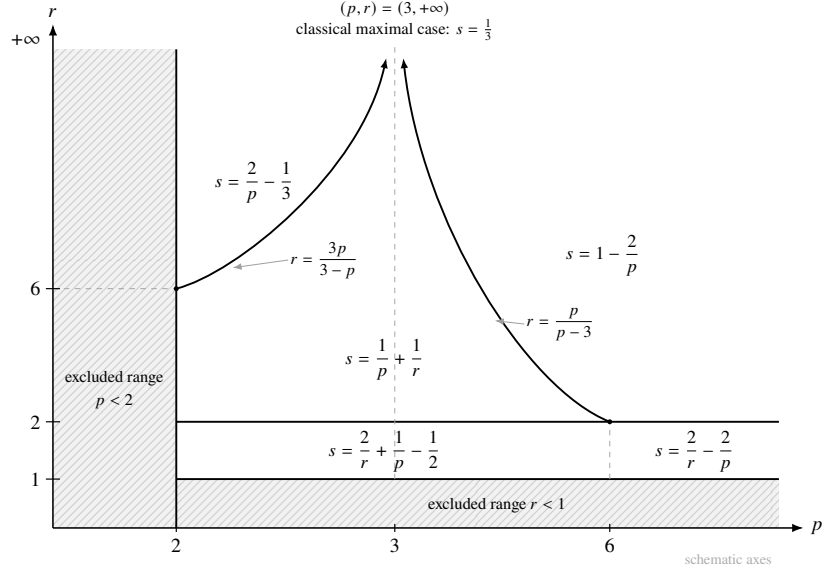

\begin{remark}\label{rek-1}
For $2\leq p\leq\infty$, define
$$m(p)=
	\begin{cases}
		\displaystyle \frac2p-\frac13,& 2\leq p\leq3,\\[4pt]
		\displaystyle 1-\frac2p,& 3\leq p\leq\infty.
	\end{cases}$$
Then
	\begin{eqnarray}\label{eq-1}
		\left\|
		\sup_{0<t<1}|e^{it\Delta}f|
		\right\|_{L^p(\mathbb{R}^2)}
		\lesssim
		\|f\|_{H^s(\mathbb{R}^2)}
	\end{eqnarray}
	holds for $s>m(p)$ and fails for $s<m(p)$ by the result of Theorem \ref{Th-1}. 	In particular, $ m(3)=1/3$,  which is exactly the Sobolev exponent of the maximal estimate of
	Du, Guth, and Li \cite{DGL2017}. Our results therefore give a sharp variational extension of their two-dimensional Schrödinger maximal estimate.
\end{remark}
\begin{remark}For $2\leq p\leq\infty$ and $1\leq r<\infty$, the main difficulty is to control the summation in the time-variation norm without introducing any loss beyond the sharp Sobolev threshold $s(p,r)$. 

\begin{enumerate} \item After decomposing the space--time region into unit cubes, we must account for both the size of the local contributions and the number of large contributions above each spatial cube. Bounding each contribution separately before summing over the time indices introduces an additional power of the scale, preventing the resulting estimate from reaching the sharp threshold.

\item We obtain the key estimate on the critical curve $r=q_p$. For each fixed spatial unit cube, we rearrange the local contributions from the corresponding space--time unit cubes in decreasing order. For each dyadic number $\lambda$, we group the spatial unit cubes according to the number of space--time unit cubes above them whose local contributions are at least $\lambda$. The bound on the number of selected cubes above each spatial cube controls the geometric density of their union.\item Combining the fractal $L^2$ estimate of Du--Zhang \cite{DZ2019} with the maximal estimate of Du--Guth--Li \cite{DGL2017}, we obtain bounds for the weak $\ell^{q_p}$ norms of the rearranged sequences. A further summation gives the strong $V^{q_p}$ estimate with logarithmic losses at the frequency-localized scale. These losses are absorbed by the strict inequality
$s>2/p-1/3$, and monotonicity gives the same regularity threshold for $q_p<r<\infty$.\end{enumerate}
\end{remark}

Furthermore, we also investigate the inhomogeneous variation operator (this definition also see \cite[Section 8]{DLTT2008}) to obtain the following conclusion:
\begin{theorem}\label{co--2}
	Let $2\leq p\leq\infty$ and $1\leq r\leq\infty$.  If $s> s(p,r)$, then
	$$
	\left\|\dot {V}^r_{[0,1]}\{e^{it\Delta}f\}\right\|_{L^p(\mathbb{R}^2)}
	\lesssim_{p,r,s}\|f\|_{H^s(\mathbb{R}^2)},
	$$
	where $\dot V^r_{[0,1]}\{e^{it\Delta}f(x)\}$ is defined as 
	$$\dot V^r_{[0,1]}\{e^{it\Delta}f(x)\}:=\sup_{t\in [0,1]}|e^{it\Delta}f(x)|+ {V}^r_{[0,1]}\{e^{it\Delta}f(x)\}. $$ 
	The estimate fails whenever  $s< s(p,r)$. 
\end{theorem}
By  \cite[Theorem 1.5]{Du2017T} and Theorem \ref{Th-1}, we have 
\begin{eqnarray*}
	\left\|\dot {V}^r_{[0,1]}\{e^{it\Delta}f\}\right\|_{L^p(\mathbb{R}^2)}
	\lesssim_{p,r,s} \left\|
	\sup_{0<t<1}|e^{it\Delta}f|
	\right\|_{L^p(\mathbb{R}^2)}+\left\|V^r_{[0,1]}\{e^{it\Delta}f\}\right\|_{L^p(\mathbb{R}^2)}
	\lesssim\|f\|_{H^{s}(\mathbb{R}^{2})},
\end{eqnarray*}
where $s> s(p,r)$. Therefore, it suffices to prove  Theorem \ref{Th-1} in order to establish Theorem \ref{co--2}.

Carleson \cite{C1980} raised the question of determining for which $s$ the pointwise convergence
\begin{eqnarray}\label{addeq-1}
	\lim_{t\to0}e^{it\Delta}f(x)=f(x)
\end{eqnarray}
holds almost everywhere for all $f\in H^s(\mathbb{R}^n)$. He established this result for $s\geq 1/4$ in the case \(n=1\)(see \cite{C1980}).
Dahlberg--Kenig \cite{DK1982}showed that convergence may fail for $s<1/4$ in all dimensions.
For $n\geq2$, Sj\"olin \cite{S1987} and Vega \cite{V1988}
independently proved convergence for $s>1/2$.
In two dimensions, Bourgain \cite{B1992,B1995} proved
convergence for $s>1/2-\delta$ with some $\delta>0$.
Lee \cite{L2006} improved this to $s>3/8$ by bilinear
restriction estimates.
For many years, it had seemed plausible that convergence actually holds
for $s > 1/4 $ in every dimension. It was not until 2012 that Bourgain \cite{B2013}  gave a counterexample showing that this is false in sufficiently high dimensions and in \cite{B2016}, Bourgain gave counterexamples showing that convergence
may fail if $s < n/ 2(n+1) $. In particular, for $n = 2$, convergence may fail if $s < 1/3.$ In 2017, Du--Guth--Li \cite{DGL2017} established the 
$L^3$ maximal estimate and showed that convergence holds almost everywhere for all $f\in H^s(\mathbb{R}^2)$  with $s > 1/3$ and this result is sharp up to the endpoint in $\mathbb{R}^{2}$.
 For related results on restriction and decoupling,
see Du--Li \cite{DL2019} and Li \cite{L2023}.

\begin{theoremA}\cite[Theorem 1.2]{DGL2017}\label{Th-2}
For any $s > 1/3$, the following bound holds: for any function $f\in H^s(\mathbb{R}^2)$,
\begin{equation}\label{eq:1.2}
\bigl\|\sup_{0<t\leq 1}\bigl|e^{it\Delta}f\bigr|\bigr\|_{L^3(B^{2}(0,1))}
 \lesssim_{s} \|f\|_{H^s(\mathbb{R}^2)}.
\end{equation}
\end{theoremA}

 For $n\geq3$, Bourgain \cite{B2013} obtained $s>1/2-1/(4n)$ by multilinear restriction estimates and the methods of Bourgain--Guth \cite{BG2011}.
In 2019, Du--Zhang \cite{DZ2019} proved the sharp local $L^2$ estimate for the Schr\"odinger maximal function for $s>n/(2(n+1))$, $n\geq3$, up to the endpoint.
Du--Guth--Li--Zhang \cite{DGLZ2018} improved the known bounds for the Hausdorff dimension of divergence sets. In particular, they proved almost everywhere convergence for $s>(n+1)/(2(n+2))$ when $n\geq3$. Li--Miao--Yu \cite{LMY2024} established the sharp global $L_x^pL_t^q$ estimates in two dimensions on the critical line $3/p+1/q=1$, for $3<p\leq6$ and $2\leq q<\infty$, with initial data in $\dot H^{1-2/p-2/q}$.

 For $1\leq r<\infty$, finite $r$-variation implies the existence of one-sided limits. In particular, for a complex-valued function $F$ on $(0,1]$,
$$V^r_{(0,1]}F<\infty\quad\Longrightarrow\quad\lim_{t\to0^+}F(t)\ \text{exists};$$ 
Bourgain\cite{B1989} proved in ergodic theory the first
variational estimate which had originated from the regularity of Brownian motion in martingale theory \cite{L1976, PX1988} since the pointwise convergence may not hold for any non-trivial dense subclass of functions in some ergodic models and thus the maximal Banach principle does not work, see \cite{B1988a,B1988b} for further results in this direction.
 
On the other hand, numerous maximal inequalities in harmonic analysis—including those for maximal singular integrals, Radon-type maximal operators, the Carleson–Hunt theorem, Cauchy-type integrals, and dimension-free Hardy–Littlewood maximal estimates—have been strengthened to variational inequalities. Mirek–Stein–Trojan \cite{MST2017} established variational estimates for discrete Radon operators and transferred these results to polynomial ergodic averages and truncated ergodic singular integrals. For Fourier partial sums, Oberlin–Seeger–Tao–Thiele–Wright \cite{OSTTW2012} proved a variational Carleson theorem with a sharp exponent range. Further developments include dimension-free variational estimates for averaging operators \cite{BMSTW2018}, bounds for spherical averages \cite{BORS2022}, and variational Fourier restriction estimates \cite{K2019}. Variational estimates also admit geometric applications: Mas–Tolsa \cite{MT2014} characterized the uniform rectifiability of Ahlfors–David regular measures via the $L^2$ boundedness of the r-variation of the Riesz transform for $r>2$.  A large body of additional work on singular integrals and averaging operators further enriches the literature on this topic; see \cite{CJRW2000,CJRW2003,CDH2021,CDHL2018,CHL2023,JSW2008,MSZ2020b}.

We sketch the proof of  Theorem \ref{Th-1}.
\begin{enumerate} \item 
Littlewood--Paley decomposition and parabolic scaling
reduce the problem to frequency-localized estimates on
$B^2(0,R)\times[0,R]$, with
$$\operatorname{supp}\mathcal{F}_x(f)\subset\{\xi\in\mathbb{R}^2:1/2\leq|\xi|\leq2\}.$$
We divide this region into unit space--time cubes$Q_j\times I_m$, where $I_m=[m,m+1]$.The spatial $L^p$ norm of the variation is controlled  by the $\ell_j^p(\ell_m^r)$ norms of the local suprema of the solution and its time derivative.A space--time reproducing formula reduces these bounds to estimates for the local $L^2$ quantities
$$f_{j,m}:=\left(\int_{Q_j\times I_m}|e^{it\Delta}f(x)|^2 dxdt\right)^{1/2},$$
taken over a fixed enlargement of the region.

\item
For each spatial cube $Q_j$, we arrange the sequence $\{f_{j,m}\}_m$ in decreasing order. The rearranged sequence relates the size of the local contributions to the number of time indices at which they exceed a given level. For dyadic $\lambda$ and $K$, let
$$E_{\lambda,K}:=\left\{Q_j:\#\{m:f_{j,m}\geq\lambda\}\approx K\right\},$$
and let $X_{\lambda,K}$ be the union of the cubes $Q_j\times I_m$ with $Q_j\in E_{\lambda,K}$ and $f_{j,m}\geq\lambda$. Thus there are approximately $K$ selected space--time cubes above each $Q_j\in E_{\lambda,K}$, as illustrated
in Figure \ref{fig:X-lambda-K}.
\begin{figure}[H]
	\centering
	\includegraphics[width=0.75\textwidth]{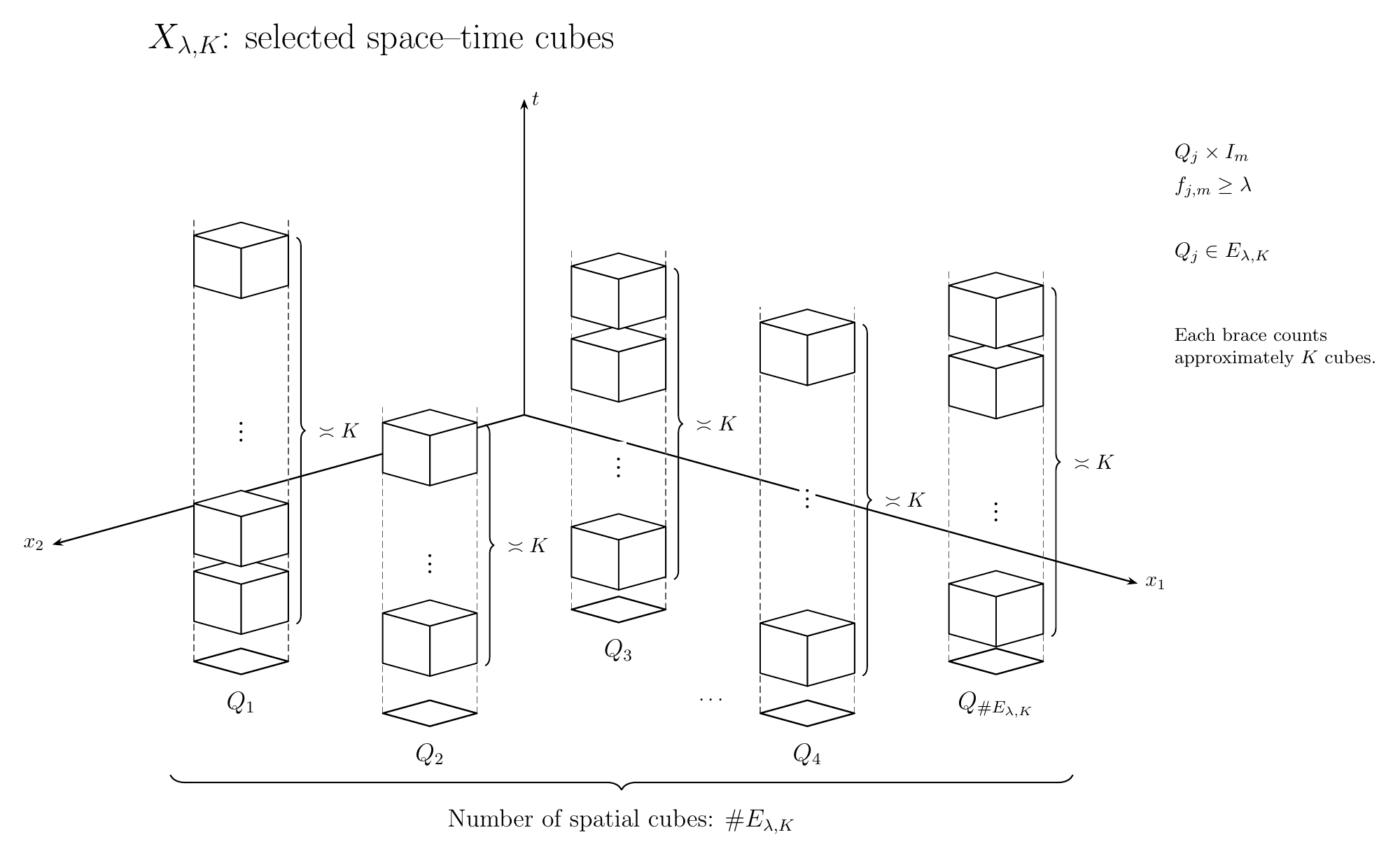}
	\caption{Schematic representation of $X_{\lambda,K}$.
		Above each $Q_j\in E_{\lambda,K}$ there are
		approximately $K$ selected space--time cubes
		$Q_j\times I_m$ satisfying $f_{j,m}\geq\lambda$.}
	\label{fig:X-lambda-K}
\end{figure}
Since each spatial unit cube $Q_j$ has at most $O(K)$ selected space--time cubes above it, we obtain the required geometric density bound for $X_{\lambda,K}$. Combining the fractal $L^2$ estimate of Du--Zhang \cite{DZ2019} with the maximal estimate of Du--Guth--Li \cite{DGL2017} gives the dependence $K^{1/q_p}$ needed to control the weak $\ell^{q_p}$ norm of the rearranged sequence. Summation over the dyadic levels and the rearrangement indices yields the strong $\ell^{q_p}$ bound with logarithmic losses, which are absorbed into $R^\varepsilon$. This proves the local estimate at $r=q_p$. Monotonicity then gives the estimates for $q_p<r<\infty$, with the same Sobolev regularity threshold as the global maximal estimate.

\item
For $2\leq p<3$ and $2<r<q_p$, interpolation with the $V^2$ estimate gives the threshold $1/p+1/r$. For $1\leq r\leq2$, the square-function estimates,
interpolation between the spatial $L^2$ and $L^6$ bounds, and Bernstein's inequality give the two regularity branches separated by $p=6$. In the range $3<p<6$ and $r>2$, we use local smoothing, square-function, and mixed space--time estimates, together with a variation inequality and interpolation; see \cite{LRS2012,LRS2013,LMY2024,R2008,RS2010} for the underlying estimates. The cases $p=3$ and $p\geq6$ follow from the corresponding maximal and square-function bounds. Spatial localization and Littlewood--Paley summation
then give the global Sobolev estimates.

\item For the necessary conditions, Bourgain's example gives the first branch of
$m(p)$ \cite{B2016}; the second follows from a scaling example of the type
used in space--time necessary conditions \cite[Section 3]{LRS2013}.  The remaining necessary conditions do not follow directly from the classical counterexamples for the Schr\"odinger maximal operator.We adapt the space--time constructions of Lee--Rogers--Seeger \cite{LRS2013} and the variational example of Guo--Roos--Yung \cite{GRY2020} by modifying
the frequency localization, spatial regions, and choice of times to obtain the necessary lower bounds on Sobolev regularity for the Schr\"odinger variation estimates.  These constructions yield the lower bounds
$$ m(p),\qquad \frac{1}{p}+\frac{1}{r},\qquad
\frac{2}{r}+\frac{1}{p}-\frac{1}{2},\qquad \frac2r-\frac{2}{p}, $$
whose maximum matches the regularity threshold in Theorem \ref{Th-1}.
\end{enumerate}

\noindent\textit{Ogranzation.} The paper is organized as follows:  In Section 2, we  give the elementary properties of variation and the space--time localization estimates of maximal Schr\"odinger operator. In Section 3, we establish the frequency-localized bounds in distinct  $(p,r)$ parameter ranges.  In Section 4  we establish the local-to-global principle and
derive the  estimates in Theorem \ref{Th-1}. In Section 5, we construct counterexamples and establish the necessary conditions for our main result.


\noindent\textit{Notation.} Throughout this paper, \begin{enumerate} \item \(A\lesssim B\) means that $A\leq CB$ for some
constant $C>0$, whose dependence on parameters is indicated by subscripts
when necessary. We write $A\approx B$ if $A\lesssim B$ and $B\lesssim A$. \item For a set $E$, $\# E$ denotes its cardinality. 
\item $\mathcal{S}$ denotes the Schwartz space. We denote by $C_c^\infty(\mathbb{R}^{n})$ the space of infinitely differentiable functions
with compact support in $\mathbb{R}^{n}$. \item $B^n(0,R)$ denotes the ball in $\mathbb{R}^n$ centered at the origin with radius $R$.\item We use $L_x^p$ and $L_t^r$ to denote the $L^p$ and $L^r$ norms with respect to the spatial variable $x$ and the time variable $t$, respectively,
for $1 \leq p,r \leq \infty$. 
\item The Fourier transforms with respect to the spatial,
temporal, and space-time variables are defined, respectively, by
$$ \left\{
\begin{aligned}
		&\mathcal{F}_{x,t}F(\xi,\tau)
	:= \int_{\mathbb{R}^{n+1}}
	e^{-i(x\cdot\xi+t\tau)}F(x,t)dxdt;\\
	&\mathcal{F}_x F(\xi,t)
	:= \int_{\mathbb{R}^n}
	e^{-ix\cdot\xi}F(x,t)dx;\\
	&\mathcal{F}_t F(x,\tau)
	:= \int_{\mathbb{R}}
	e^{-it\tau}F(x,t)dt.
	\end{aligned}
\right.
$$
\end{enumerate}
	\section{Preliminaries}

	\begin{lemma}\cite[Riesz--Thorin theorem]{BL1976}\label{le-8}
		Let $T$ be a bounded linear
		operator from $L^{p_{0}} (X, A,\mu)$ to $L^{q_{0}} (Y , B, \nu)$ with norm $M_{0}$ and from $L^{p_{1}} (X, A,\mu)$ to
		$L^{q_{1}} (Y , B, \nu)$ with norm $M_{1}$. Then, $T$ is bounded from $L^{p_{\theta}} (X, A,\mu)$ to $L^{q_{\theta}} (Y , B, \nu)$
		with norm $M_{\theta}$ such that
		$$ M_{\theta} \leq M_{0}^{1-\theta} M_{1}^{\theta} $$
		with $\theta\in(0,1)$  and
		$$
		\left\{\begin{array}{ll}
			\frac{1}{p_{\theta}}=\frac{1-\theta}{p_{0}}+\frac{\theta}{p_{1}} \\
			\frac{1}{q_{\theta}}=\frac{1-\theta}{q_{0}}+\frac{\theta}{q_{1}}
		\end{array}\right.
		$$
	\end{lemma}

	\begin{lemma}\cite[Corollary 1.7]{DZ2019} \label{le--1}
		Let $n\geq1$ and \(B_l\) is a unit cubes centered at $x_l\in\mathbb{R}^{n+1}$. For any $\varepsilon>0$, there exists a constant $C_{\varepsilon}$
		such that the following holds for all $R \geq 1$ and all $f$ with $supp \mathcal{F}_x(f) \subset B^{n}(0,1)$.
		Suppose that $X=\bigcup_{l}B_{l}$ is a union of lattice unit cubes in $B^{n+1}(0,R)$. Let
		$1 \leq \alpha\leq n+1$ and $\gamma$ be given by
		\begin{eqnarray*}
			\gamma:= \sup_{B^{n+1}(x',r)\subset B^{n+1}(0,R)} \frac{\#\{ B_{l}\in X:B_{l}\subset B(x',r)\}}{r^{\alpha}},
		\end{eqnarray*}
		where $x'\in \mathbb{R}^{n+1}$ and $r\geq 1$. Then
		\begin{eqnarray*}
			\|e^{it\Delta}f\|_{L^{2}(X)}\lesssim  \gamma^{\frac{1}{n+1}}R^{\frac{\alpha}{2(n+1)}+\varepsilon}\|f\|_{2}.
		\end{eqnarray*}
	\end{lemma}

	\begin{lemma}\cite[Theorem 1.3]{DGL2017}\label{le--2}
		For $\varepsilon>0$, there exists a constant $C_{\varepsilon}$ such that
		$$\left\| \sup_{0< t\leq R}|e^{it\Delta}f(x)|\right\| _{L^{3}(B^{2}(0,R))}\lesssim R^{\varepsilon}\|f\|_{2}.$$
		holds for all $R\geq 1$ and $supp \mathcal{F}_x(f) \subset \{\xi\in \mathbb{R}^{2}:|\xi|\approx 1\}$.
	\end{lemma}
	\begin{lemma}\label{co--1}
		For $\varepsilon>0$, there exists a constant $C_{\varepsilon}$ such that
		$$\left\| \sup_{-4R\leq t\leq 4R}|e^{it\Delta}f(x)|\right\| _{L^{3}(B^{2}(0,4R))}\lesssim R^{\varepsilon}\|f\|_{2}.$$
		holds for all $R\geq 1$ and $supp \mathcal{F}_x(f) \subset \{\xi\in \mathbb{R}^{2}:|\xi|\approx 1\}$.
	\end{lemma}
	\begin{proof}
		By a direct computation and a change of variables in Lemma \ref{le--2}, we can finish the proof. Therefore, we can omit the detail.
	\end{proof}

    \section{Local variational estimates for Schr\"odinger equation.}
    \subsection{The proof of the critical $r=\frac{3p}{3-p}$ with $2\leq p<3$}
	Let $Q_{j}\subset \mathbb{R}^{2}$ and $(x_{1,j},x_{2,j})\in \mathbb{Z}^{2}$ where $Q_{j}$  is a unit cube which is  centered at $(x_{1,j},x_{2,j})$ and $	I_{m}:=[m,m+1]$ with $m\in \mathbb{Z}$. Define 
    \begin{eqnarray}\label{eq-4}
 	f_{j,m}:=\left( \int_{B_{j,m}} |u(x,t)|^{2}dxdt\right)^{1/2}=\left( \int_{B_{j,m}} |e^{it\Delta} f(x)|^{2}dxdt\right)^{1/2}, 
    \end{eqnarray}
    where $B_{j,m}:=Q_{j}\times I_{m}$.
    \begin{lemma}\label{le--3}
    	Let $R,p,r\geq 1$ $f\in \mathcal{S}(\mathbb{R}^{2})$ and $supp\ \mathcal{F}_x(f) \subset\{ \xi\in \mathbb{R}^{2}: 1/2 \leq|\xi|\leq 2\}$. We have
	    \begin{eqnarray*}
	    \|V^r_{[0,R]}\{e^{it\Delta}f \}\|_{L^p(B^{2}(0,R))}\lesssim\left( \sum_{Q_{j} \cap B^{2}(0,4R)\neq \emptyset}\left( \sum_{I_{m}\cap[-4R,4R]\neq \emptyset} |f_{j,m}|^{r}\right)^{p/r} \right)^{1/p}+O(R^{-1000})\|f\|_{2}.
	    \end{eqnarray*}
    \end{lemma}
    \begin{proof}
	    Set
    	$\mathscr{S}:=\bigl\{	j\in\{1,\ldots,J\}:	t_{j-1},t_j\in I_{m}\text{ for some } I_{m}\cap[0,R]\neq \emptyset\bigr\},
	    \mathscr{H}:=\{1,\ldots,J\}\setminus\mathscr{S}.$ For $\mathscr{S}$, we have 
        \begin{eqnarray*}
        \sum_{j\in\mathscr{S}}
        \bigl|e^{it_{j}\Delta} f(x)-e^{it_{j-1}\Delta}f(x)\bigr|^r&\leq&\sum_{I_{m}\cap [0,R]\neq \emptyset}  \left(  \int_{I_{m}}|\partial_{t} e^{it\Delta} f(x)|dt\right)^r\\
        &\lesssim&  \sum_{I_{m}\cap [0,R]\neq \emptyset}\left( \sup_{t\in I_{m}} |\partial_te^{it\Delta}f(x)|\right)^{r},
        \end{eqnarray*}
	    On the other hand, for $\mathscr{H}$,  we have
	    $$ |e^{it_{j+1}\Delta} f(x)-e^{it_{j}\Delta} f(x)|^{r}\lesssim |e^{it_{j+1}\Delta} f(x)|^{r}+|e^{it_{j}\Delta} f(x)|^{r}.$$
        Thus, we have 
        \begin{eqnarray*}
	    \sum_{j\in\mathscr{H}}
	    \bigl|e^{it_{j}\Delta} f(x)-e^{it_{j-1}}f(x)\bigr|^r    
	     &\lesssim&  \sum_{I_{m}\cap [0,R]\neq \emptyset}\left( \sup_{t\in I_{m}} |e^{it\Delta}f(x)|\right)^{r}.
        \end{eqnarray*}
	    Therefore, we can deduce
\begin{eqnarray}\label{eq-6}
    V_{[0,R]}^{r}\{e^{it\Delta}f(x)\} 
	\lesssim
	\left(\sum_{I_{m}\cap [0,R]\neq \emptyset}\left(\sup_{t\in I_{m}}\left(|e^{it\Delta}f(x)|+|\partial_t e^{it\Delta}f(x)|\right)\right)^r\right)^{1/r}.
\end{eqnarray}
        which implies 

        \begin{eqnarray*}
	    &&\|V^r_{[0,R]}\{e^{it\Delta}f \}\|_{L^p(B^{2}(0,R))}\\
	    &\lesssim&
	    \left[\sum_{\substack{Q_{j}\cap B^{2}(0,R)\neq \emptyset}}\int_{Q_j}
	    \left(\sum_m \sup_{t\in I_m}(|e^{it\Delta}f(x)|+|\partial_te^{it\Delta}f(x)|)^r\right)^{p/r}dx\right]^{1/p}\\
	    &\leq&
	    \left( \sum_{Q_{j} \cap B^{2}(0,R)\neq \emptyset }\left( \sum_{I_{m}\cap [0,R]\neq \emptyset}\bigg(\sup_{(x,t) \in B_{j,m}} \Big(|e^{it\Delta}f(x)|+|\partial_{t}e^{it\Delta}f(x)|\Big)\ \bigg)^{r}\right)^{p/r} \right)^{1/p}.
        \end{eqnarray*}

        Since the space--time Fourier transform of $e^{it\Delta}f$ is supported on the paraboloid $\tau=-|\xi|^2$, and $\operatorname{supp}\mathcal{F}_x(f)\subset\{\xi\in\mathbb{R}^2:1/2\leq|\xi|\leq2\}$, we have
        $$\operatorname{supp}\mathcal{F}_{x,t}(e^{it\Delta}f) \subset S
        :=\left\{(\xi,\tau)\in\mathbb{R}^2\times\mathbb{R}:
        \frac12\leq|\xi|\leq2,\ \tau=-|\xi|^2\right\}.$$
        Choose $\psi\in C_c^\infty$such that $\psi=1$ on an open neighbourhood of $S$, and set$\eta=\mathcal F_{x,t}^{-1}\psi,$ which indicates
    	\begin{eqnarray}\label{eq--1}
		\sup_{(x,t)\in B_{j,m}}|e^{it\Delta}f(x)|&\leq&\sum_{Q_{j'},I_{m'}}\sup_{(x,t)\in B_{j,m}}\left|\int_{B_{j',m'}}e^{is\Delta}f(w)\eta(x-w,t-s)dwds\right|\\
		&\leq&\sum_{Q_{j'},I_{m'}}\sup_{(x,t)\in B_{j,m}}\|\eta(x-\cdot,t-\cdot) \|_{L^{2}(B_{j',m'})} \|e^{it\Delta}f\|_{L^{2}(B_{j',m'})}.\notag
	\end{eqnarray}
	Define
	$$ \tilde{f}_{j,m}=
	\left\{
	\begin{array}{rcl}
		&&f_{j,m}, \quad B_{j,m}\cap B^{2}(0,4R)\times [-4R,4R]\neq \emptyset ;\\
		&&0, \quad\quad \text{others},
	\end{array}
	\right.
	$$
	where $f_{j,m}$ is given in (\ref{eq-4}). By \eqref{eq--1} and the  property of Schwartz function, 	for any $M>0$, we obtain
	\begin{eqnarray*}
		\sup_{(x,t)\in B_{j,m}} |e^{it\Delta} f(x)|&\lesssim& \sum_{Q_{j'},I_{m'}}\sup_{(x,t)\in B_{j,m}}\|\eta(x-\cdot,t-\cdot) \|_{L^{2}(B_{j',m'})} \|e^{it\Delta}f\|_{L^{2}(B_{j',m'})}\\
		&\lesssim&	 \sum_{Q_{j'},I_{m'}} (1+|(j'-j,m'-m)|)^{-M} {f}_{ j', m'}\\
		&\lesssim& \sum_{Q_{j'},I_{m'}} (1+|(j-j',m-m')|)^{-M}\tilde{f}_{j',m'}+A,
	\end{eqnarray*}
	where 
	$$A:=\sum_{Q_{j'},I_{m'}}(1+|(j-j',m-m')|)^{-M}	\bigl(f_{j',m'}-\tilde f_{j',m'}\bigr). $$ 
For the main term, we have
\begin{eqnarray}\label{eq-5}
	&&
	\bigg( \sum_{Q_{j}\cap B^{2}(0, R)\neq \emptyset}\bigg( \sum_{I_{m}\cap [0,R]\neq \emptyset}\bigg(\sum_{Q_{j'},I_{m'}} (1+|(j-j',m-m')|)^{-M}\tilde{f}_{j',m'}\bigg) ^{r}\bigg) ^{p/r}\bigg) ^{1/p}\\
	&=&\bigg( \sum_{Q_{j}\cap B^{2}(0, R)\neq \emptyset}\bigg( \sum_{I_{m}\cap [0,R]\neq \emptyset}\bigg( \sum_{Q_{j'},I_{m'}} (1+|(j',m')|)^{-M}\tilde{f}_{j-j',m-m'}\bigg) ^{r}\bigg) ^{p/r}\bigg) ^{1/p}\notag\\
	&\leq&\notag
	\sum_{Q_{j'},I_{m'}} (1+|(j',m')|)^{-M}\bigg( \sum_{Q_{j}\cap B^{2}(0, R)\neq\emptyset}\bigg(\sum_{I_{m}\cap[0,R]\neq\emptyset}|\tilde{f}_{j-j',m-m'}| ^{r}\bigg) ^{p/r}\bigg) ^{1/p}\\
	&\leq&\bigg( \sum_{Q_{j}\subset \mathbb{R}^{2}}\bigg(   \sum_{I_{m}\subset \mathbb{R}}|\tilde{f}_{j,m}| ^{r}\bigg) ^{p/r}\bigg) ^{1/p}\bigg( \sum_{Q_{j'},I_{m'}} (1+|(j',m')|)^{-M}\bigg)\notag \\
	&\lesssim&\bigg( \sum_{Q_{j}\subset \mathbb{R}^{2}}\bigg(   \sum_{I_{m}\subset \mathbb{R}}|\tilde{f}_{j,m}| ^{r}\bigg) ^{p/r}\bigg) ^{1/p}.\notag
	\end{eqnarray}
For the remainder term $A$, if $Q_{j'}\times I_{m'} \cap  B^{2}(0,4R)\times [-4R,4R] =\emptyset$, then we  have $(1+|j-j',m-m'|)^{-1}\lesssim (1+R)^{-1}$. And $Q_{j'}\times I_{m'}$ is a unit cube, then we have $1+|(x-w,t-s)|\approx 1+|(j-j',m-m')|$.  Thus, by the   property of $\eta$, we gives
	\begin{eqnarray*}
		A&\leq& \sum_{Q_{j'}\times I_{m'} \cap B^{2}(0,4R)\times [-4R,4R]= \emptyset} \sum_{|(j-j',m-m')|\gtrsim R}(1+|(j-j',m-m')|)^{-M} f_{j',m'}\\
		&\lesssim& \|f\|_2\sum_{|(j-j',m-m')|\gtrsim R}	(1+|(j-j',m-m')|)^{-M}
		\lesssim_M R^{3-M}\|f\|_2.
	\end{eqnarray*}
 We observe that there are at most $O(R^{2})$ sets $Q_{j}$ contained in $ B^{2}(0,4R)$ and $O(R)$ sets $ I_{m}$ contained in $[0,R+2]$. Then we get
 \begin{eqnarray*}
 \bigg( \sum_{Q_{j} \cap B ^{2} (0,R)\neq \emptyset }\bigg( \sum_{I_{m}\cap [0,R+2]\neq \emptyset} A^{r}\bigg) ^{p/r}\bigg) ^{1/p}\lesssim (1+R)^{-1000}\|e^{it\Delta}f\|_{2}=(1+R)^{-1000}\|f\|_{2}.
 \end{eqnarray*}
 For $\partial_{t} e^{it\Delta}f(x)$, we also apply the  Schwartz bump function $\eta$. Note that
 $$\partial_{t} e^{it\Delta}f(x)=   e^{it\Delta}f\ast (\partial_{t}\eta)(x)$$
 and $\partial_{t}\eta $ still is a Schwartz  function. Therefore, the proof of $\partial_{t} e^{it\Delta}f(x) $ is similar to $e^{it\Delta}f(x) $, and we omit the details.
\end{proof}

\begin{lemma}\label{th--1}
	Let  $2\leq p<3$, $\frac{3p}{3-p}\leq r<\infty$, $R\geq 1$, $f\in \mathcal{S}(\mathbb{R}^{2})$ and $supp\ \mathcal{F}_x(f) \subset\{ \xi\in \mathbb{R}^{2}: 1/2 \leq|\xi|\leq 2\}  $. There holds
	 \begin{eqnarray*}
	 	\left( \sum_{Q_{j} \cap B^{2}(0,4R)\neq \emptyset}\left( \sum_{I_{m}\cap [-4R,4R]\neq \emptyset} |f_{j,m}|^{r}\right)^{p/r} \right)^{1/p}\lesssim R^{2/p-2/3+\varepsilon} \|f\|_{2},
	 \end{eqnarray*}
where $f_{j,m}$ is given by $ \ref{eq-4}$.
\end{lemma}

\begin{proof}
	Let $\lambda$ and $K$ be the dyadic numbers. Define
\begin{eqnarray}\label{adeq--1}
	\left\{
	\begin{array}{rcl}
		&&M(\lambda,Q_{j}):= \{m\in [-4R-4,4R+4]\cap \mathbb{Z}: f_{j,m}\geq \lambda \} ;\\
		&&N(\lambda,Q_{j}):=\# M(\lambda,Q_j);\\
		&&E_{\lambda,K}:=\{Q_{j}: Q_{j}\cap B^{2}(0,4R)\neq \emptyset \text{ and } N(\lambda,Q_{j})\approx K\};\\
		&&X_{\lambda,K}:=\bigcup_{Q_{j}\in E_{\lambda,K}}\bigcup_{m\in M(\lambda,Q_j)} Q_{j}\times I_{m}.
	\end{array}
	\right.
\end{eqnarray}
We regroup $f_{j,m}$ according to the number of unit cubes in $B^{2}(0,4R)\times [-4R-4,4R+4]$ and the size of $f_{j,m}$ over each cube. In fact, we ignore  so small $\lambda$ that $\lambda<R^{-100}\|f\|_{2}$. Thus we have $\#\{\lambda: R^{-100}\|f\|_{2}\leq \lambda \leq  R^{10}\|f\|_{2}\}\approx \log R$. The $K$ implies the number of $I_{m}$ which perpendicular to $Q_{j}.$ For $r\geq \frac{3p}{3-p}$, we can deduce
\begin{eqnarray}\label{eq--2}
	\left(\sum_{I_{m}\cap[-4R-4,4R+4]\neq\emptyset}|f_{j,m}|^{r}\right)^{\frac{1}{r}}&=&\left(\sum_{I_{m}\cap[-4R-4,4R+4]\neq\emptyset}|f_{j,m}|^{r-\frac{3p}{3-p}}|f_{j,m}|^{\frac{3p}{3-p}}\right)^{\frac{1}{r}}\notag\\
	&\leq&\left(\sum_{I_{m}\cap[-4R-4,4R+4]\neq\emptyset}|f_{j,m}|^{\frac{3p}{3-p}}\right)^{\frac{(3-p)}{3p}-\frac{1}{r}}\left(\sum_{I_{m}\cap[-4R-4,4R+4]\neq \emptyset}|f_{j,m}|^{\frac{3p}{3-p}}\right)^{\frac{1}{r}}\notag\\
	&=&\left(\sum_{I_{m}\cap[-4R-4,4R+4]\neq\emptyset}|f_{j,m}|^{\frac{3p}{3-p}} \right)^{\frac{3-p}{3p}}.
\end{eqnarray}
Since the terms with $m\le-1$ are treated in the same way by the change of variable, it suffices to consider the terms with $m\ge0$.

To obtain the local estimates on the critical curve, we   combine weak-type bounds for  rearrangements in the time variable. Define $$f_{Q_{j}}(1)\geq f_{Q_{j}}(2)\geq \ldots \geq f_{Q_{j}}([4R]+4),$$
where $f_{Q_{j}}(k)$ denotes the $k$-th largest value of $f_{j,m}$ over all unit cube $Q_{j}\times I_{m}$ perpendicular to $Q_{j}$. For any fixed  $l\in \mathbb{Z}$ with $1\leq l\leq [4R]+4$, there exists a dyadic number $K$ such that $K\leq l \leq 2K$. Thus there holds
$$ f_{Q_{j}}(l)\leq f_{Q_{j}}(K)\leq F(Q_{j}) K^{-1/p+1/3}\leq 2^{1/p-1/3} F(Q_{j})l^{-1/p+1/3}, $$
where
$$ F(Q_{j}):=\sup_{4R+4\geq K\geq 1} K^{1/p-1/3}f_{Q_{j}}(K).$$
Then we have
\begin{eqnarray}\label{eq--6}
	\left( 	\sum_{I_{m}\cap [0,4R+4]\neq \emptyset} f_{j,m}^{3p/(3-p)}\right)^{1/p-1/3} &=&  \left( \sum_{l=1}^{[4R]+4} f_{Q_{j}}(l)^{3p/(3-p)}\right)^{1/p-1/3}\\
	&\leq& 2^{1/p-1/3}F(Q_{j})\bigg( \sum_{l=1}^{[4R]+4} l^{-1}\bigg)^{1/p-1/3}\notag\\
	&\lesssim& O(\log R) F(Q_{j}).\notag
\end{eqnarray}
For any dyadic number $K\leq 4R+4$, it is easy to see
\begin{eqnarray}\label{eq--7}
	f_{Q_{j}}(K)\geq \lambda\Rightarrow N(\lambda,Q_{j})\geq K.
\end{eqnarray}
Set $$Q_{\mu}:= \{Q_{j} \cap B^{2}(0,4R)\neq \emptyset: F(Q_{j})\geq \mu\}.$$
For every $f_{j,m}$, there exists a uniform upper bound $\|f\|_{2}$. Therefore, we have
$$ F(Q_{j})=\sup_{K\leq 4R+4} K^{1/p-1/3}f_{Q_{j}}(K)\lesssim R^{1/p-1/3}\|f\|_{2}.$$
If  $\mu> F(Q_{j})$ for any $Q_{j}$, we have $|Q_{\mu}|=0$. By the layer-cake Theorem, we can deduce
\begin{eqnarray}\label{adeq--3}
	\sum_{Q_{j} \cap B^{2}(0,4R)\neq \emptyset} F(Q_{j})^{p}&=&p\int_{0}^{\infty}\mu^{p-1}\left(\sum_{ Q_{j} \cap B^{2}(0,4R)\neq \emptyset} \chi_{\{\mu\leq F(Q_{j}) \}} \right) d\mu\\
	&\leq&p\int_{0}^{CR^{1/p-1/3}\|f\|_{2}}\mu^{p-1}|\{Q_{j} \cap B^{2}(0,4R)\neq \emptyset:\mu\leq F(Q_{j}) \}| d\mu\notag\\
	&=&p\int_{0}^{\mu_{0}} \mu^{p-1}|\{Q_{j} \cap B^{2}(0,4R)\neq \emptyset:\mu\leq F(Q_{j}) \}| d\mu\notag\\
	&&+ p\int_{\mu_{0}}^{C R^{1/p-1/3}\|f\|_{2}} \mu^{p-1}|\{Q_{j}\cap B^{2}(0,4R)\neq \emptyset:\mu\leq F(Q_{j}) \}| d\mu\notag\\
	&\leq& p\int_{0}^{\mu_{0}} \mu^{p-1}R^{2} d\mu+ p\int_{\mu_{0}}^{C R^{1/p-1/3}\|f\|_{2}} \mu^{p-1}|Q_{\mu}| d\mu.\notag
\end{eqnarray}
  Then, for $\mu_{0}=R^{-2/3+\varepsilon} \|f\|_{2}$, we have 
$$ \left( \int_{0}^{\mu_{0}} \mu^{p-1}R^{2} d\mu \right) ^{1/p}\lesssim R^{2/p-2/3+\varepsilon}\|f\|_{2}. $$
Thus, we assume $0<\mu\leq F(Q_{j})$.
Then we have
\begin{eqnarray*}
	Q_{\mu}&=& \{Q_{j} \cap B^{2}(0,4R)\neq \emptyset:  \sup_{4R\geq K\geq 1} K^{1/p-1/3}f_{Q_{j}}(K)\geq \mu\}\\
	&\subset&\bigcup_{K} \{Q_{j} \cap B^{2}(0,4R)\neq \emptyset: K^{1/p-1/3}f_{Q_{j}}(K )\geq \mu\}
\end{eqnarray*}
For    $F(Q_{j})\geq \mu>\mu_{0}$ , there exists a $K$ such that
$$ K^{1/p-1/3}f_{Q_{j}}( K)\geq \mu,  $$
which indicates
$$f_{Q_{j}}(K )\geq K ^{1/3-1/p}\mu.$$
Then we can find a dyadic number $\lambda_{K}$  such that  $2\lambda_{K} >K ^{1/3-1/p}\mu\geq\lambda_{K}$ ( $\mu_0$ ensures that the number of $\lambda_K$ terms remains $O(\log R)$). Thus, we can deduce
\begin{eqnarray}\label{addeq--4}
	Q_{\mu}\subset\bigcup_{K}\{Q_{j} \cap B^{2}(0,4R)\neq \emptyset : N(\lambda_{K},Q_{j})\geq K \}=\bigcup_{K}\tilde{E}_{\lambda_{K},K },
\end{eqnarray}
where
$$\tilde{E}_{\lambda,K}:= \{Q_{j} \cap B^{2}(0,4R)\neq \emptyset: N(\lambda,Q_{j})\geq K\}.$$
For each fixed $Q_{j}\in E_{\lambda,K}$, the strip perpendicular to $Q_{j}$ contains $\approx K$ tubes. Therefore, we can get
\begin{eqnarray}\label{eq--3}
	\left| X_{\lambda,K} \right|\approx  K |E_{\lambda,K}|\approx K \#E_{\lambda,K}.
\end{eqnarray}
By \eqref{eq--3}, we have
\begin{eqnarray*}
	\|e^{it\Delta}f(x)\|_{L^{2}(X_{\lambda,K})}^{2}&=&\sum_{Q_{j}\in E_{\lambda,K}}\sum_{m\in M(\lambda,Q_{j})} f_{j,m}^{2}\\
	&\gtrsim& \lambda^{2}\left( \sum_{Q_{j}\in E_{\lambda,K}}N(\lambda,Q_{j}) \right)\\
	&\gtrsim&\lambda^{2}K |E_{\lambda,K}|.
\end{eqnarray*}
As for the parameters appearing in Lemma \ref{le--1}, we may assume without loss of generality that $n=2$, $\alpha=2$ and $X=X_{\lambda,K}$. Then we have
\begin{eqnarray*}
	\#\{B_{l}\subset X_{\lambda,K}:B_{l}\subset B(x',r) \} \lesssim Kr^{2},
\end{eqnarray*}
which indicates
\begin{eqnarray*}
	\gamma \lesssim \frac{K r^{2}}{r^{2}}= K.
\end{eqnarray*}
Thus, for $E_{\lambda,K}$, we can apply Lemma \ref{le--1} to deduce
\begin{eqnarray*}
	\lambda K^{1/2}|E_{\lambda,K}|^{1/2}\lesssim \|e^{it\Delta} f(x)\|_{L^{2}(X_{\lambda,K})}\lesssim K^{1/3} R^{1/3+\varepsilon} \|f\|_{2}.
\end{eqnarray*}
For $Q_{j}\in E_{\lambda,K}$, there exists an $m\in M(\lambda,Q_{j})$ such that
\begin{eqnarray*}
	\lambda^{2}&\leq& \int_{Q_{j}\times I_{m}}|e^{it\Delta} f(x)|^{2} dxdt\\
	&\leq& \int_{Q_{j}}(\sup_{-4R\leq t\leq 4R}|e^{it\Delta}f(x)|)^{2}dx.
\end{eqnarray*}
Note that the number of $Q_{j}\in  E_{\lambda,K}$ is equal to $|E_{\lambda,K}|$. Then for all $Q_{j}\in E_{\lambda,K}$, we have
\begin{eqnarray*}
	\lambda^{2}|E_{\lambda,K}|
	&\leq& \int_{\cup_{Q_{j} \in E_{\lambda,K}}}(\sup_{-4R\leq t\leq 4R}|e^{it\Delta}f(x)|)^{2}dx.
\end{eqnarray*}
By the H\"older inequalities and Lemma \ref{co--1}, we get
\begin{eqnarray*}
 	\lambda^{2}|E_{\lambda,K}|&\lesssim& \left\| \sup_{-4R\leq t\leq 4R}|e^{it\Delta}f(x)|\right\| _{L^{3}(B^{2}(0,4R))}^{2} \left| \bigcup_{Q_{j} \in E_{\lambda,K}} Q_{j}\right|^{1/3}\\
    &\lesssim& R^{\varepsilon}\|f\|_{2}^{2} |E_{\lambda,K}|^{1/3},
\end{eqnarray*}
which implies
$$ \lambda |E_{\lambda,K}|^{1/3} \lesssim R^{\varepsilon} \|f\|_{2}.$$
To obtain the $L^{p}$ estimates of variational Schr\"odinger operator, we apply Lemma \ref{le-8}:
\begin{eqnarray*}
  \lambda K^{\theta/6}|E_{\lambda,K}|^{(1-\theta)/3+\theta/2}\leq R^{\theta/3+\varepsilon} \|f\|_{2},
\end{eqnarray*}
there $\theta$ satisfies
\begin{eqnarray*}
	 \frac{1}{p}=\frac{\theta}{2}+\frac{1-\theta}{3}\Rightarrow \theta =\frac{6}{p}-2.
\end{eqnarray*}
Thus we have
$$\lambda K^{1/p-1/3}|E_{\lambda,K}|^{1/p}\lesssim R^{2/p-2/3+\varepsilon} \|f\|_{2},$$
which implies
\begin{eqnarray}\label{eq--4}
	|E_{\lambda,K}| \lesssim \left( \frac{R^{2/p-2/3+\varepsilon}\|f\|_{2}}{\lambda K^{1/p-1/3}}\right) ^{p}.
\end{eqnarray}
Then we get
$$\tilde{E}_{\lambda,K}\subset \bigcup_{\{l\in \mathbb{Z}_{+}:\left[  \log_{2}\left(\frac{4R+4}{K} \right)\right] +1 \geq l\geq 0 \}} E_{\lambda,2^{l}K},$$
where upper bound of $l$ ensures $2^{l}K$ can attain all possible dyadic numbers. By \eqref{eq--4}, we have
\begin{eqnarray}\label{eq--5}
	\bigg|  \tilde{E}_{\lambda,K}\bigg| \leq \sum_{l\geq 0}|E_{\lambda,2^{l}K}|\leq  \sum_{l\geq 0}\left( \frac{R^{2/p-2/3+\varepsilon}\|f\|_{2}}{\lambda (2^{l}K)^{1/p-1/3}} \right)^{p}.
\end{eqnarray}
Following $p\in[2,3)$, we have
        \begin{eqnarray*}
         \bigg|  \tilde{E}_{\lambda,K}\bigg| \lesssim   \left( \frac{R^{2/p-2/3+\varepsilon}\|f\|_{2}}{\lambda K^{1/p-1/3}} \right)^{p}.
        \end{eqnarray*}
For the fixes $Q_{j}$, we need to rearrange the $f_{j,m}$ in the tubes perpendicular to $Q_{j}$ according to their sizes.   By \eqref{eq--7}, we have $N(\lambda_{K},Q_{j})\geq K $. 

Note  that $4R+4\geq K\geq 1$, there are at most $O(\log R)$ possible values of $K$ satisfying this condition. Then, by \eqref{eq--5} and \eqref{addeq--4}, we can deduce
\begin{eqnarray}\label{eq--8}
 |Q_{\mu}|&\leq& \sum_{K}   |\tilde{E}_{\lambda_{K},K}|\\
 &\lesssim&\sum_{K}\left( \frac{R^{2/p-2/3+\varepsilon}\|f\|_{2}}{\lambda_{K} K^{1/p-1/3}} \right)^{p}\notag\\
 &\lesssim&O(\log R)\left( \frac{R^{2/p-2/3+\varepsilon}\|f\|_{2}}{  \mu} \right)^{p}.\notag
\end{eqnarray}
Also, take $\mu_{0}=R^{-2/3+\varepsilon}\|f\|_{2}$. By the estimate \eqref{adeq--3} and \eqref{eq--8}, we can deduce
 \begin{eqnarray}\label{eq--9}
 	\left( \sum_{Q_{j}\cap B^{2}(0,4R)\neq \emptyset} F(Q_{j})^{p}\right)^{1/p} &\lesssim& \mu_{0}R^{2/p+\varepsilon}+R^{2/p-2/3+\varepsilon}\|f\|_{2}\left( \int_{\mu_{0}}^{R^{1/p-1/3}\|f\|_{2}} \mu^{-1}d\mu\right) ^{1/p}\notag\\
 	&\lesssim& R^{2/p-2/3+\varepsilon}\|f\|_{2}+ R^{2/p-2/3+\varepsilon}\|f\|_{2}\log R^{1/p+1/3-\varepsilon}\notag\\
 	&\lesssim& R^{2/p-2/3+\varepsilon}\|f\|_{2}.
 \end{eqnarray}
 Therefore, by \eqref{eq--2}, \eqref{eq--6} and \eqref{eq--9}, we have
 $$	\left( \sum_{Q_{j}\cap B^{2}(0,4R)\neq \emptyset}\left( \sum_{I_{m}\cap [-4R-4,4R+4]\neq \emptyset} |f_{j,m}|^{r}\right)^{p/r} \right)^{1/p}\lesssim R^{2/p-2/3+\varepsilon} \|f\|_{2}$$
 for $r\geq\frac{3p}{3-p}$.
\end{proof}
\begin{lemma}\label{Le-2}
	Let $2\leq p <3, \frac{3p}{3-p}\leq r\leq \infty$, $R\geq 1$  and $f\in \mathcal{S}(\mathbb{R}^{2}),$
	where $supp\ \mathcal{F}_x(f) \subset \{ \xi\in \mathbb{R}^{2}: 1/2 \leq|\xi|\leq 2\}  $. There holds
	$$\|V_{[0,R]}^{r}\{e^{it\Delta}f\}\|_{L^{p}(B^{2}(0,R))}
	\lesssim  R^{2/p-2/3+\varepsilon} \|f\|_{2}. $$
\end{lemma}
\begin{proof}
	By Lemma \ref{le--3}  and Lemma \ref{th--1}, we have
	\begin{eqnarray*}
	&&\|V_{[0,R]}^{r}\{e^{it\Delta}f(x)\}\|_{L^{p}(B^{2}(0,R))}\\
	&\lesssim& \left( \sum_{Q_{j} \cap B^{2}(0,4R)\neq \emptyset}\left( \sum_{I_{m}\cap [0,R+2]\neq \emptyset} \Big(\sup_{(x,t) \in B_{j,m}}( |e^{it\Delta}f(x)|+|\partial_{t}e^{it\Delta}f(x)|)\ \Big)^{r}\right)^{p/r} \right)^{1/p}\\
	 &\lesssim& \left( \sum_{Q_{j} \cap B^{2}(0,4R)\neq \emptyset}\left( \sum_{I_{m}\cap [0,R+2]\neq \emptyset} |f_{j,m}|^{r}\right)^{p/r} \right)^{1/p}+O(R^{-1000})\|f\|_{2}\\
	&\lesssim& R^{2/p-2/3+\varepsilon} \|f\|_{2}.
	\end{eqnarray*}
\end{proof}

\begin{lemma}\label{Le-3}
	Let $2\leq p\leq\infty$, $R\geq1$, and let
	$f\in\mathcal S(\mathbb{R}^2)$ satisfy
	$$\operatorname{supp}\mathcal{F}_x(f)\subset\left\{\xi\in\mathbb{R}^2:\frac12\leq|\xi|\leq 2\right\}.$$
	Then
	\begin{eqnarray*}
		\left\|V_{[0,R]}^2\{e^{it\Delta}f\}\right\|_{L^p(B^2(0,R))}\lesssim_p R^{	\max\{-\frac14+\frac{3}{2p},0\}}\|f\|_2.
	\end{eqnarray*}
\end{lemma}
\begin{proof}
	Let
	$$I_m=[m,m+1],$$
	where the sum below is taken over all $m$ such that $I_m\cap[0,R+1]\neq\varnothing$. By \eqref{eq-6},
	\begin{eqnarray}\label{eq--10}
		V_{[0,R]}^2\{e^{it\Delta}f(x)\}\lesssim\left(\sum_m\left[\sup_{t\in I_m}\left(|e^{it\Delta}f(x)|+|\partial_t e^{it\Delta}f(x)|\right)\right]^2\right)^{1/2}.
	\end{eqnarray}
	For $\mathcal{F}_{t}\left[e^{it\Delta}f\right] $, we have
	\begin{eqnarray*}
		H:=\operatorname{supp}\mathcal{F}_{t}\left[e^{it\Delta}f\right]\subset[-4,-\frac{1}{4}].
	\end{eqnarray*}
 Similarly to the proof of \eqref{eq--1},   by introducing a smooth bump function $\psi\in \mathcal{S}(\mathbb{R})$ and $	\mathcal{F}_{t}\left[ \psi\right] |_{H}=1$, we can obtain
	\begin{eqnarray*}
		|e^{it\Delta}f(x)|+|\partial_t e^{it\Delta}f(x)|&\leq&\int_{\mathbb{R}}\left(|\psi(t-s)|+|\psi'(t-s)|\right) |e^{is\Delta}f(x)|ds\\
		&\lesssim&\int_{\mathbb{R}}	(1+|s-m|)^{-M}|e^{is\Delta}f(x)|  ds
	\end{eqnarray*}
for every $t\in I_m$. By the Cauchy--Schwarz inequality,
	\begin{eqnarray}	\label{eq--11}
		\sup_{t\in I_m}\left(|e^{it\Delta}f(x)|+|\partial_t e^{it\Delta}f(x)|\right)^2\lesssim\int_{\mathbb{R}}(1+|s-m|)^{-M}|e^{is\Delta}f(x)|^2ds.
 	\end{eqnarray}
	Summing \eqref{eq--11} in $m$ and using
	\eqref{eq--10}, we obtain
	\begin{eqnarray}\label{eq--12}
		V_{[0,R]}^2\{e^{it\Delta}f(x)\}\lesssim_M\left(\int_{\mathbb{R}}\bigg(\sum_{\substack{m\in\mathbb{Z}\\I_m\cap[0,R]\neq\varnothing}}(1+|s-m|)^{-M}\bigg)|e^{is\Delta}f(x)|^2ds\right)^{1/2},
	\end{eqnarray}
	where
	$$\sum_{\substack{m\in\mathbb{Z}\\I_m\cap[0,R]\neq\varnothing}}(1+|s-m|)^{-M}$$
	satisfies
	\begin{eqnarray}\label{eq--13}
		\left\{
		\begin{array}{rcl}
			&\sup_{s\in\mathbb{R}}\sum_{\substack{m\in\mathbb{Z}\\I_m\cap[0,R]\neq\varnothing}}(1+|s-m|)^{-M}\lesssim C;\\
			&\int_{\mathbb{R}}\sum_{\substack{m\in\mathbb{X}\\I_m\cap[0,R]\neq\varnothing}}(1+|s-m|)^{-M}ds\lesssim R.
		\end{array}
		\right.
	\end{eqnarray}

	We first estimate the right-hand side of \eqref{eq--12} in $L_x^2$. By Fubini's theorem,
	\begin{eqnarray*}
		&&\left\|\left(\int_{\mathbb{R}}\bigg(\sum_{\substack{m\in\mathbb{Z}\\I_m\cap[0,R]\neq\varnothing}}(1+|s-m|)^{-M}\bigg)|e^{is\Delta}f(x)|^2ds\right)^{1/2}\right\|_{L_x^2(B^{2}(0,R))}^2\\
		&\leq&\int_{\mathbb{R}}\sum_{\substack{m\in\mathbb{Z}\\I_m\cap[0,R]\neq\varnothing}}(1+|s-m|)^{-M}  \|e^{is\Delta}f\|_2^2 ds\\
		&\lesssim&R\|f\|_2^2.
	\end{eqnarray*}
	Therefore, we can deduce
	\begin{eqnarray}\label{eq--14}
		\left\|V_{[0,R]}^2\{e^{it\Delta}f(x)\}\right\|_{L_x^2(B^{2}(0,R))}\lesssim R^{1/2}\|f\|_2.
	\end{eqnarray}
	We next prove the $L_x^6(B^{2}(0,R))$ estimate. By \eqref{eq--13},
	\begin{eqnarray*}
		&&\left\|\left( \int_{\mathbb{R}}\sum_{\substack{m\in\mathbb{Z}\\ I_m\cap[0,R]\neq\varnothing}}(1+|s-m|)^{-M}|e^{is\Delta}f(x)|^2ds\right)^{1/2}\right\|_{L_x^6(B^{2}(0,R))}\\
		&\lesssim&\left\|\left(\int_{\mathbb{R}}|e^{is\Delta}f(x)|^2ds\right)^{1/2}\right\|_{L_x^6(B^{2}(0,R))}.
	\end{eqnarray*}

By the endpoint mixed-norm estimate of Li, Miao and Yu\cite[Theorem 1.6]{LMY2024} and $\operatorname{supp}\mathcal{F}_x(f)\subset\left\{\xi\in\mathbb{R}^2:\frac12\leq|\xi|\leq2\right\}$, we have
\begin{eqnarray}\label{eq-add2}
	\|e^{it\Delta}f\|_{L_x^6(\mathbb{R}^2;L_t^2(\mathbb{R}))}\lesssim\|f\|_{\dot H^{-1/3}(\mathbb{R}^2)}\approx\|f\|_2.
\end{eqnarray} 
By Bernstein’s inequality and the Fourier support, the same bound holds for
$6\leq p\leq\infty$. Consequently,  for $p\geq6$, we have 
\begin{eqnarray}\label{eq-add4}
	\left\|	V_{[0,R]}^2\{e^{it\Delta}f(x)\}	\right\|_{L_x^p(B^{2}(0,R))} &\lesssim &	\left\| \left(\sum_m\left[\sup_{t\in I_m}\left(|e^{it\Delta}f(x)|+|\partial_t e^{it\Delta}f(x)|\right)	\right]^2\right)^{1/2}\right\|_{L_x^p(B^{2}(0,R))}\notag\\
	&\lesssim&\left\|\left(\int_{\mathbb{R}}|e^{it\Delta}f(x)|^2dt\right)^{1/2}\right\|_{L_x^p(B^2(0,R))}\lesssim \|f\|_2.
\end{eqnarray}
	For $2\leq p\leq6$, choose $0\leq\theta\leq1$ such that
	\begin{eqnarray*}
		\frac{1}{p}=\frac{\theta}{2}+\frac{1-\theta}{6}\Rightarrow\theta=\frac{3}{p}-\frac{1}{2}.
	\end{eqnarray*}
	By 	\eqref{eq-add4} with $p=6$, we obtain 
	\begin{eqnarray}\label{eq-add5}
		\left\|V_{[0,R]}^2e^{it\Delta}f(x)\right\|_{L_x^6(B^{2}(0,R))}\lesssim
		\|f\|_2.
	\end{eqnarray}
Thus, interpolating	\eqref{eq--14} and \eqref{eq-add5}, we find
	\begin{eqnarray}\label{eq-add1}
		&&\left\|V_{[0,R]}^2\{e^{it\Delta}f(x)\}\right\|_{L_x^p(B^{2}(0,R))} \notag\\
	    &\lesssim&\left\|\left(\int_{\mathbb{R}}\bigg(\sum_{\substack{m\in\mathbb{Z}\\I_m\cap[0,R]\neq\varnothing}}(1+|s-m|)^{-M}\bigg)|e^{is\Delta}f(x)|^2ds\right)^{1/2}\right\| _{L_x^p(B^{2}(0,R))}\notag\\
		&\lesssim&R^{\theta/2}\|f\|_2=R^{-\frac14+\frac{3}{2p}}\|f\|_2
	\end{eqnarray}
	with $2\leq p\leq6$. This completes the proof.
\end{proof}

\begin{lemma}\label{Le-4}
	Let $2\leq p <3, 2\leq r\leq \frac{3p}{3-p}\text{ and }f\in\mathcal{S}(\mathbb{R}^{2})$, where$supp\mathcal{F}_x(f)\subset\{\xi\in \mathbb{R}^{2}: 1/2 \leq|\xi|\leq 2\} $. There holds
	$$\|V_{[0,R]}^{r}\{e^{it\Delta}f(x)\}\|_{L^{p}(B^{2}(0,R))}\lesssim  R^{-\frac{1}{2}+\frac{3}{2p}+\frac{1}{2r}+\varepsilon} \|f\|_{2}. $$
\end{lemma}
\begin{proof}
Lemma \ref{Le-3} and Lemma \ref{Le-2} give the two endpoint estimates
	\begin{eqnarray*}
		\left\{
		\begin{array}{rcl}
			\left\|V_{[0,R]}^2\{e^{it\Delta}f\}\right\|_{L^p(B^2(0,R))}
			&\lesssim&R^{-\frac14+\frac{3}{2p}}\|f\|_2,\quad 3>p\geq 2;\\
			\left\|V_{[0,R]}^{q_p} \{e^{it\Delta}f\}\right\|_{L^p(B^2(0,R))}
			&\lesssim&R^{\frac2p-\frac23+\varepsilon}\|f\|_2\quad 3>p\geq 2.
		\end{array}
		\right.
	\end{eqnarray*}
	Fix $2\leq r<q_p$, and choose $0<\theta\leq1$ such that
	\begin{eqnarray}\label{eq--16}
		\frac{1}{r}=\frac{\theta}{2}+\frac{1-\theta}{q_p}.
	\end{eqnarray}
	By	Hölder's inequality, we have
	\begin{eqnarray}\label{eq--17}
		V_{[0,R]}^r\{e^{it\Delta}f(x)\}\leq\left(V_{[0,R]}^2\{e^{it\Delta}f(x)\}\right)^\theta\left(V_{[0,R]}^{q_p}\{e^{it\Delta}f(x)\}\right)^{1-\theta}.
	\end{eqnarray}
	Applying Hölder's inequality in $x$ to \eqref{eq--17}, we obtain
	\begin{eqnarray}\label{eq--18}
		\left\|V_{[0,R]}^r\{e^{it\Delta}f\}\right\|_{L^p(B^2(0,R))}&\leq&\left\|V_{[0,R]}^2\{e^{it\Delta}f\}\right\|_{L^p(B^2(0,R))}^{\theta}\left\|V_{[0,R]}^{q_p}\{e^{it\Delta}f\}\right\|_{L^p(B^2(0,R))}^{1-\theta}.
	\end{eqnarray}
	Since
	$$\frac{1}{q_p}=\frac{1}{p}-\frac{1}{3},$$
	equation \eqref{eq--16} becomes
	\begin{eqnarray*}
		\frac{1}{r}=\frac{\theta}{2}+(1-\theta)\left(\frac{1}{p}-\frac{1}{3}\right),
	\end{eqnarray*}
	which indicates
	\begin{eqnarray}\label{eq--19}
		\theta=\frac{\frac{1}{r}-\frac{1}{q_p}}{\frac{1}{2}-\frac{1}{q_p}}.
	\end{eqnarray}
	For \eqref{eq--18}, by \eqref{eq--19}, we have
	\begin{eqnarray*}
		\left\|V_{[0,R]}^r \{e^{it\Delta}f\}\right\|_{L^p(B^2(0,R))}
		&\lesssim&R^{(-\frac{1}{4}+\frac{3}{2p})\theta+(1-\theta)\left(\frac{2}{p}-\frac{2}{3} \right) +\varepsilon}\|f\|_2\\
		&\lesssim&R^{-\frac{1}{2}+\frac{3}{2p}+\frac{1}{2r}+\varepsilon}\|f\|_2.
	\end{eqnarray*}
\end{proof}

\subsection{The range $1\leq r\leq2$ and $2\leq p\leq\infty$.}
\begin{lemma}\label{Le-5}
	Let $1\leq r\leq2$, $2\leq p\leq\infty$, and $R\geq1$. Suppose that
	$f\in\mathcal{S}(\mathbb{R}^{2})$ satisfies
	$$\operatorname{supp}\mathcal{F}_x(f)\subset\left\{\xi\in\mathbb{R}^{2}:\frac{1}{2}\leq|\xi|\leq 2 \right\}.$$
	Then
	\begin{eqnarray}\label{eq--20}
		\left\|V^{r}_{[0,R]}\{e^{it\Delta}f\}\right\|_{L^{p}(B^{2}(0,R))}&\lesssim_{p,r}&\left\{\begin{array}{ll}R^{\frac{1}{r}-\frac{3}{4}+\frac{3}{2p}}\|f\|_{2},& 2\leq p\leq6,\\
		R^{\frac{1}{r}-\frac{1}{2}}\|f\|_{2},& 6\leq p\leq\infty.
		\end{array}\right.
	\end{eqnarray}
\end{lemma}

\begin{proof}
	By Lemma \ref{le--3},
	\begin{eqnarray}\label{eq--21}
		V^{r}_{[0,R]}\{e^{it\Delta}f(x)\}&\lesssim_{r}&\left(\sum_{\substack{m\in\mathbb{Z}\\I_m\cap[0,R]\neq\varnothing}}\left[\sup_{t\in I_m}\left(\left|e^{it\Delta}f(x)\right|+\left|\partial_t e^{it\Delta}f(x)\right|\right)\right]^{r}\right)^{1/r}.
	\end{eqnarray}
	There are at most $O(R)$ nonempty intervals in \eqref{eq--21}. Since $1\leq r\leq2$, the finite-dimensional embedding $\ell^{2}\subset\ell^{r}$ gives
	\begin{eqnarray}\label{eq--22}
		&&\left(\sum_{\substack{m\in\mathbb{Z}\\I_m\cap[0,R]\neq\varnothing}}\left[\sup_{t\in I_m}\left(\left|e^{it\Delta}f(x)\right|+\left|\partial_t e^{it\Delta}f(x)\right|\right)\right]^{r}\right)^{1/r}
		\notag\\
		&\lesssim_{r}&R^{\frac{1}{r}-\frac{1}{2}}\left(\sum_{\substack{m\in\mathbb{Z}\\I_m\cap[0,R]\neq\varnothing}}\left[\sup_{t\in I_m}\left(\left|e^{it\Delta}f(x)\right|+\left|\partial_t e^{it\Delta}f(x)\right|\right)\right]^{2}\right)^{1/2}.
	\end{eqnarray}
	Combining \eqref{eq--21} and \eqref{eq--22}, and then take the
	$L^{p}$ norm over $B^{2}(0,R)$, we obtain
	\begin{eqnarray}\label{eq--23}
		\left\|V^{r}_{[0,R]}\{e^{it\Delta}f\}\right\|_{L^{p}(B^{2}(0,R))}&\lesssim&R^{\frac{1}{r}-\frac{1}{2}}	\left\|\left(
		\sum_{\substack{m\in\mathbb{Z}\\I_m\cap[0,R]\neq\varnothing}}\left[\sup_{t\in I_m}\left(\left|e^{it\Delta}f\right|+\left|\partial_t e^{it\Delta}f\right|\right)\right]^{2}\right)^{1/2}
		\right\|_{L^{p}(B^{2}(0,R))}.
	\end{eqnarray}
  	For $2\leq p$,	we apply the estimates of \eqref{eq-add1}  and \eqref{eq-add4} to obtain 
	\begin{eqnarray}\label{eq--24-1}
		&&\left\|\left(\sum_{\substack{m\in\mathbb{Z}\\I_m\cap[0,R]\neq\varnothing}}\left[\sup_{t\in I_m}\left(\left|e^{it\Delta}f\right|+\left|\partial_t e^{it\Delta}f\right|\right)\right]^{2}\right)^{1/2}\right\|_{L^{p}(B^{2}(0,R))}\lesssim R^{\max\{-\frac{1}{4}+\frac{3}{2p},0\}}\|f\|_{2}
	\end{eqnarray}
Therefore, substituting \eqref{eq--24-1} into \eqref{eq--23} gives
	\begin{eqnarray*}
		\left\|V^{r}_{[0,R]}\{e^{it\Delta}f\}\right\|_{L^{p}(B^{2}(0,R))}&\lesssim_{p,r}&
		\left\{\begin{array}{ll}
		R^{\frac{1}{r}-\frac{3}{4}+\frac{3}{2p}}
		\|f\|_{2},& 2\leq p\leq6,\\
		R^{\frac{1}{r}-\frac{1}{2}}\|f\|_{2},& 6\leq p\leq\infty,
		\end{array}\right.
	\end{eqnarray*}
which finish the proof of $\eqref{eq--20}$.
\end{proof}

\subsection{The range $3<p<6$ and $2<r<\infty$}

\begin{lemma}\label{Le-7}
	Let $3<p<6$, $2<r<\infty$, $R\geq 1$, and let
	$f\in\mathcal{S}(\mathbb{R}^{2})$ satisfy
	$$\operatorname{supp}\mathcal{F}_x(f)\subset\left\{\xi\in\mathbb{R}^{2}:\frac{1}{2}\leq |\xi|\leq 2\right\}.$$
	Then, for every $\varepsilon>0$,
	\begin{eqnarray}\label{eq--25}
		\left\|V^{r}_{[0,R]}\{e^{it\Delta}f\}\right\|_{L^{p}(B^{2}(0,R))}&\lesssim&R^{\max\left\{0,-\frac{1}{2}+\frac{3}{2p}+\frac{1}{2r}\right\}+\varepsilon}\|f\|_{2}.
	\end{eqnarray}
\end{lemma}

\begin{proof}
	Define the critical variation exponent
	\begin{eqnarray}\label{eq--26}
		q_{p}'=\frac{p}{p-3},\qquad\frac{1}{q_{p}'}=1-\frac{3}{p}.
	\end{eqnarray}
	Since $3<p<6$, it follows that $q_{p}'>2$. We use two estimates established previously. First, the $2$-variation estimate gives
	\begin{eqnarray}\label{eq--27}
		\left\|V^{2}_{[0,R]}\{e^{it\Delta}f\}\right\|_{L^{p}(B^{2}(0,R))}&\lesssim_{p,\varepsilon}&R^{-\frac{1}{4}+\frac{3}{2p}+\varepsilon}\|f\|_{2}.
	\end{eqnarray}
	Second, by \cite[Theorem 1.6]{LMY2024}, together with the Fourier support of $f$, gives
	$$\|e^{it\Delta}f\|_{L_x^pL_t^{q_{p}'}(\mathbb{R}^2\times\mathbb{R})}\lesssim_p\|f\|_{\dot H^{4/p-1}(\mathbb{R}^2)}\lesssim_p \|f\|_2.$$
	Applying the same estimate to $i\Delta f$ and using $\partial_t e^{it\Delta}f=e^{it\Delta}(i\Delta f)$, we also obtain
$$\|\partial_t e^{it\Delta}f\|_{L_x^pL_t^{q_{p}'}(\mathbb{R}^2\times\mathbb{R})}\lesssim_p\|\Delta f\|_{\dot H^{4/p-1}(\mathbb{R}^2)}\lesssim_p \|f\|_2.$$
	
For almost every fixed $x$, the variation inequality\cite[(39)]{JSW2008}, applied to the function $t\mapsto e^{it\Delta}f(x)$, yields
	$$V^{q_{p}'}_{[0,R]}\{e^{it\Delta}f(x)\}\lesssim_p\|e^{it\Delta}f(x)\|_{L_t^{q_{p}'}(\mathbb{R})}^{1-1/q_{p}'}\|\partial_t e^{it\Delta}f(x)\|_{L_t^{q_{p}'}(\mathbb{R})}^{1/q_{p}'}.$$
	Take the $L^p$ norm in $x$ and applying H\"older's inequality,
	we conclude that
	\begin{eqnarray}\label{eq--28}
		\left\|V^{q_p'}_{[0,R]}\{e^{it\Delta}f\}\right\|_{L^p(B^2(0,R))}&\lesssim_p&\|e^{it\Delta}f\|_{L_x^pL_t^{q_p'}(\mathbb{R}^{3})}^{1-1/q_p'}\|\partial_t e^{it\Delta}f\|_{L_x^pL_t^{q_p'}(\mathbb{R}^{3})}^{1/q_p'}
		\notag\\
		&\lesssim_p& \|f\|_2.
	\end{eqnarray}
	We divide the argument into two cases: 
	\medskip\noindent
	\textbf{Case 1: $2<r<q_{p}'$.}	Choose $0<\theta<1$ such that
	\begin{eqnarray}\label{eq--29}
		\frac{1}{r}=\frac{\theta}{2}+\frac{1-\theta}{q_{p}'}.
	\end{eqnarray}
	For any finite increasing sequence
	$$0\leq t_{0}<t_{1}<\cdots<t_{M}\leq R,$$
	apply the interpolation inequality between $\ell^{2}$ and $\ell^{q_{p}'}$ to the corresponding increment sequence. Take the supremum over all such sequences yields	the pointwise variation inequality
	\begin{eqnarray*}
		V^{r}_{[0,R]}\{e^{it\Delta}f(x)\}&\leq&\left(V^{2}_{[0,R]}\{e^{it\Delta}f(x)\}\right)^{\theta}\left(V^{q_{p}'}_{[0,R]}\{e^{it\Delta}f(x)\}\right)^{1-\theta}.
	\end{eqnarray*}
	Applying H\"older's inequality in the spatial variable and then
	using \eqref{eq--27} and \eqref{eq--28}, we obtain
	\begin{eqnarray}\label{eq--30}
		\left\|V^{r}_{[0,R]}\{e^{it\Delta}f\}\right\|_{L^{p}(B^{2}(0,R))}&\leq&\left\|V^{2}_{[0,R]}\{e^{it\Delta}f\}\right\|_{L^{p}(B^{2}(0,R))}^{\theta}\left\|V^{q_{p}'}_{[0,R]}\{e^{it\Delta}f \}\right\|_{L^{p}(B^{2}(0,R))}^{1-\theta}\notag\\
		&\lesssim&R^{\theta\left(-\frac{1}{4}+\frac{3}{2p}\right)+\varepsilon}\|f\|_{2}.
	\end{eqnarray}

	It remains to simplify the exponent of $R$. By \eqref{eq--26},
	\begin{eqnarray*}
		-\frac{1}{4}+\frac{3}{2p}&=&\frac{1}{2}\left(\frac{1}{2}-\frac{1}{q_{p}'}\right).
	\end{eqnarray*}
	Moreover, \eqref{eq--29} implies
	\begin{eqnarray*}
		\theta\left(\frac{1}{2}-\frac{1}{q_{p}'}\right)&=&\frac{1}{r}-\frac{1}{q_{p}'}.
	\end{eqnarray*}
	Consequently,
	\begin{eqnarray}\label{eq--31}
		\theta\left(-\frac{1}{4}+\frac{3}{2p}\right)=\frac{1}{2}\left(\frac{1}{r}-\frac{1}{q_{p}'}\right)=-\frac{1}{2}+\frac{3}{2p}+\frac{1}{2r}.
	\end{eqnarray}
	Substituting \eqref{eq--31} into \eqref{eq--30} gives
	\begin{eqnarray*}
		\left\|V^{r}_{[0,R]}\{e^{it\Delta}f\}\right\|_{L^{p}(B^{2}(0,R))}&\lesssim_{p,r,\varepsilon}&R^{-\frac{1}{2}+\frac{3}{2p}+\frac{1}{2r}+\varepsilon}\|f\|_{2},\qquad 2<r<q_{p}'.
	\end{eqnarray*}

	\medskip
	\noindent
	\textbf{Case 2: $r\geq q_{p}'$.}

	The monotonicity of variation seminorms gives
	$$V^{r}_{[0,R]}\{e^{it\Delta}f(x)\}\leq V^{q_{p}'}_{[0,R]}\{e^{it\Delta}f(x)\}.$$
	Therefore, \eqref{eq--28} implies
	$$\left\|V^{r}_{[0,R]}\{e^{it\Delta}f\}\right\|_{L^{p}(B^{2}(0,R))}\lesssim_{p,r,\varepsilon}R^{\varepsilon}\|f\|_{2}.$$
	Notice that $r\geq q_{p}'$ is equivalent to
	$$-\frac{1}{2}+\frac{3}{2p}+\frac{1}{2r}\leq 0.$$
	Combining the two cases proves \eqref{eq--25}.
\end{proof}

\subsection{The endpoint ranges $p=3$ and $p\geq 6$}

\begin{lemma}\label{Le-6}
	Let $2<r\leq\infty$, $R\geq1$, and let
	$f\in\mathcal{S}(\mathbb{R}^{2})$ satisfy
	$$\operatorname{supp}\mathcal{F}_x(f) \subset\left\{\xi\in\mathbb{R}^{2}:\frac{1}{2}\leq|\xi|\leq2\right\}.$$
	Then, for every $\varepsilon>0$,
	\begin{eqnarray}\label{eq--32}
		\left\|V^{r}_{[0,R]}\{e^{it\Delta}f\}\right\|_{L^{p}(B^{2}(0,R))}&\lesssim_{p,r,\varepsilon}&
	    \left\{\begin{array}{ll}
			R^{\frac{1}{2r}+\varepsilon}\|f\|_{2},& p=3,\\
			\|f\|_{2},& 6\leq p\leq\infty.
		\end{array}\right.
	\end{eqnarray}
\end{lemma}

\begin{proof}
	We first consider $p=3$. The previously established
	$2$-variation estimate and the sharp maximal estimate give,
	respectively,
	\begin{eqnarray}\label{eq--33}
		\left\{\begin{array}{ll}
			\left\|V^{2}_{[0,R]}\{e^{it\Delta}f\}\right\|_{L^{3}(B^{2}(0,R))}\lesssim_{\varepsilon}R^{\frac{1}{4}+\varepsilon}\|f\|_{2};\\
			\left\|\sup_{t\in [0,R]}|e^{it\Delta}f|
			\right\|_{L^{3}(B^{2}(0,R))}\lesssim_{\varepsilon}R^{\varepsilon}\|f\|_{2},
		\end{array}\right.
	\end{eqnarray}
	where maximal estimate is given by \cite[Theorem 1.3]{DGL2017}. The second estimate follows from the sharp Schr\"odinger maximal
	estimate proved previously, together with
	$$V^{\infty}_{[0,R]}\{e^{it\Delta}f(x)\}=2\sup_{0\leq t\leq R}
	\left|e^{it\Delta}f(x)\right|.$$

	For $r>2$, set
	\begin{eqnarray*}
		\theta=\frac{2}{r},\qquad\frac{1}{r}=\frac{\theta}{2}+\frac{1-\theta}{\infty}.
	\end{eqnarray*}
	As in the proof above, interpolation of the increment sequence
	gives
	$$V^{r}_{[0,R]}\{e^{it\Delta}f(x)\}\lesssim\left(V^{2}_{[0,R]}\{e^{it\Delta}f(x)\}\right)^{\frac{2}{r}}\left(\sup_{t\in [0,R]}|e^{it\Delta}f(x)|\right)^{1-\frac{2}{r}}.$$
	Using H\"older's inequality in $x$ and \eqref{eq--33}, we
	obtain
	\begin{eqnarray*}
		\left\|V^{r}_{[0,R]}\{e^{it\Delta}f\}\right\|_{L^{3}(B^{2}(0,R))}\lesssim_{r,\varepsilon}\left(R^{\frac{1}{4}+\varepsilon}\|f\|_{2}\right)^{\frac{2}{r}}\left(R^{\varepsilon}\|f\|_{2}\right)^{1-\frac{2}{r}}\lesssim R^{\frac{1}{2r}+\varepsilon}\|f\|_{2}.
	\end{eqnarray*}
	This proves the first estimate in \eqref{eq--32}.

	We next consider $6\leq p\leq\infty$. The $2$-variation estimate
	proved in the preceding subsection gives
	\begin{eqnarray}\label{eq--34}
		\left\|V^{2}_{[0,R]}\{e^{it\Delta}f\}\right\|_{L^{p}(B^{2}(0,R))}&\lesssim_{p}&\|f\|_{2}.
	\end{eqnarray}
	Since $r>2$, monotonicity of the variation seminorm implies
$$V^{r}_{[0,R]}\{e^{it\Delta}f(x)\}\leq V^{2}_{[0,R]}\{e^{it\Delta}f(x)\}.$$
	Combining this pointwise inequality with \eqref{eq--34}, we obtain
	\begin{eqnarray*}
		\left\|V^{r}_{[0,R]}\{e^{it\Delta}f\}\right\|_{L^{p}(B^{2}(0,R))}&\lesssim_{p,r}&\|f\|_{2},\qquad6\leq p\leq\infty.
	\end{eqnarray*}
	This proves the second estimate in \eqref{eq--32}.
\end{proof}

\section{The proof of the sufficiency part in Theorem \ref{Th-1}}
    In this section, we establish the Theorem \ref{Th-3} that allows us to deduce global variational boundedness from local variational boundedness. Combining this result with the local estimates established in Lemmas \ref{Le-2}-  \ref{Le-6} of Section 3, we prove the sufficiency part of our main theorem. The subsequent spatial globalization follows the near--far localization method in   Lee,   Rogers and  Seeger \cite{LRS2013,LRV2011,R2008}. Since these references concern maximal Schrödinger operator, we give the details needed for the variation norm.

\begin{theorem}\label{Th-3}
	Let $2\leq p\leq\infty$, $1\leq r<\infty$, and $\beta\geq0$.
	Suppose that, for every $\varepsilon>0$, every $R\geq1$, and every
	$g\in\mathcal S(\mathbb{R}^2)$ satisfying
	\begin{eqnarray}\label{eq--35}
		\operatorname{supp}\mathcal{F}_x(g)\subset\left\{\xi\in\mathbb{R}^2:
		\frac{1}{2}\leq|\xi|\leq2\right\},
	\end{eqnarray}
	one has
	\begin{eqnarray}\label{eq--36}
		\left\|V^r_{[0,R]}\left\lbrace e^{it\Delta}g\right\rbrace\right\|_{L^p(B^2(0,R))}
		\lesssim_{p,r,\varepsilon}R^{\beta+\varepsilon}\|g\|_2.
	\end{eqnarray}
     Then, for every
	\begin{eqnarray*}
		s>1-\frac{2}{p}+2\beta,
	\end{eqnarray*}
	one has
	\begin{eqnarray}\label{eq--37}
		\left\|V^r_{[0,1]}\left\lbrace e^{it\Delta}f\right\rbrace \right\|_{L^p(\mathbb{R}^2)}\lesssim_{p,r,s}\|f\|_{H^s(\mathbb{R}^2)}.
	\end{eqnarray}
\end{theorem}

\begin{proof}
	First, we apply a Littlewood--Paley decomposition and parabolic rescaling to the $r$-variation Schr\"odinger operator. 	Define $g(x):=R^{-1}g_R(x/R).$ For $R\geq1$, we can obtain  $g_R\in\mathcal S(\mathbb{R}^2)$, $\operatorname{supp}\mathcal{F}_x(g_R)\subset\left\{\xi\in\mathbb{R}^2:\frac {R}{2}\leq|\xi|\leq2R\right\}$ and $\mathcal{F}_x(g)(\xi)=R\mathcal{F}_x(g_R)(R\xi)$. Moreover, $\|g\|_2^2=\|g_R\|_2^2.$ A direct change of variables in the Schr\"odinger operator gives
	\begin{eqnarray*}
		e^{it\Delta}g_R(x)=R\left[e^{iR^2t\Delta}g\right](Rx).
	\end{eqnarray*}
	Hence, we have
	\begin{eqnarray*}
		V^r_{[0,1]}\left\lbrace e^{it\Delta}g_R(x)	\right\rbrace =R V^r_{[0,R^2]}\left\lbrace e^{it\Delta}g(Rx)\right\rbrace.
	\end{eqnarray*}
	If $x\in B^2(0,R)$, then $Rx\in B^2(0,R^2)$. Then, we obtain
	\begin{eqnarray*}
		\left\|V^r_{[0,1]}\left\lbrace 	e^{it\Delta}g_R\right\rbrace \right\|_{L^p(B^2(0,R))}=R^{1-2/p}\left\|V^r_{[0,R^2]}\left\lbrace e^{it\Delta}g\right\rbrace \right\|_{L^p(B^2(0,R^2))}.
	\end{eqnarray*}
	We now apply the assumed estimate \eqref{eq--36} with $R$ replaced by $R^2$. Thus,
	\begin{eqnarray*}
		\left\|V^r_{[0,1]}\left\lbrace e^{it\Delta}g_R\right\rbrace\right\|_{L^p(B^2(0,R))}&=&R^{1-2/p}\left\|V^r_{[0,R^2]}\left\lbrace e^{it\Delta}g\right\rbrace\right\|_{L^p(B^2(0,R^2))}\\
		&\lesssim&R^{1-2/p}R^{2\beta+2\varepsilon}\|g\|_2\\
		&=&R^{1-2/p+2\beta+2\varepsilon}\|g \|_2.
	\end{eqnarray*}
	Replacing $2\varepsilon$ by $\varepsilon$, we have
	\begin{eqnarray*}
		\left\|V^r_{[0,1]}\left\lbrace e^{it\Delta}g_R\right\rbrace\right\|_{L^p(B^2(0,R))}\lesssim_{p,r,\varepsilon}R^{1-2/p+2\beta+\varepsilon}\|g \|_2.
	\end{eqnarray*}
	Let
	$\widetilde g_R(x):=g_R(x+x_0).$
	We have
	\begin{eqnarray*}
		e^{it\Delta}\widetilde g_R(x)=\left(e^{it\Delta}g_R\right)(x+x_0)
	\end{eqnarray*}
	and
	\begin{eqnarray*}
		\mathcal{F}_x(\widetilde g_R)(\xi)=e^{ix_0\cdot\xi}\mathcal{F}_x(g_R)(\xi),
	\end{eqnarray*}
	which indicates
	\begin{eqnarray*}
		\operatorname{supp}\mathcal{F}_x(\widetilde g_R)=\operatorname{supp}\mathcal{F}_x(g_R),\qquad\|\widetilde g_R\|_2=\|g_R\|_2.
	\end{eqnarray*}
	Therefore, we can obtain
	\begin{eqnarray}\label{eq--38}
		\left\|V^r_{[0,1]}\left\lbrace e^{it\Delta}g_R \right\rbrace\right\|_{L^p(B^2(x_0,R))}&=&\left\|V^r_{[0,1]}\left\lbrace e^{it\Delta}\widetilde g_R\right\rbrace \right\|_{L^p(B^2(0,R))}\\
		&\lesssim&R^{1-2/p+2\beta+\varepsilon}\|g \|_2.\notag
	\end{eqnarray}
	uniformly for all $x_0\in\mathbb{R}^2$.

	 Following the ideas in \cite{LRS2013} and \cite{LRV2011}, choose
	$a\in C_c^\infty(\mathbb{R}^2)$	such that
	\begin{eqnarray*}
		a(\xi)=1\qquad\text{for }1/2\leq|\xi|\leq 2
	\end{eqnarray*}
	and
	\begin{eqnarray}\label{eq-3}
		\operatorname{supp}a\subset\left\{\xi\in\mathbb{R}^2:1/4\leq|\xi|\leq4\right\}.
	\end{eqnarray}
	Define
	\begin{eqnarray}\label{eq-2}
		\mathcal{F}_x(T_{a}g_R)=a(\xi/R)\mathcal{F}_x(g_R)=\mathcal{F}_x(g_R)
	\end{eqnarray}
	Partition $\mathbb{R}^2$ into cubes
	\begin{eqnarray*}
		Q_j:=R\bigl(j+[0,1)^2\bigr),\qquad j\in\mathbb{Z}^2.
	\end{eqnarray*}
	Let $\{\chi_j\}_{j\in\mathbb{Z}^2}$ be a smooth partition of unity
	adapted to these cubes, so that
	\begin{eqnarray*}
		\sum_{j\in\mathbb{Z}^2}\chi_j(x)=1.
	\end{eqnarray*}
	Define
	\begin{eqnarray*}
		g_{R,j}:=T_{a}(\chi_{j}g_{R})=\mathcal{F}_x^{-1}(a(\xi/R)\mathcal{F}_x(\chi_{j}g_{R})).
	\end{eqnarray*}
	By (\ref{eq-2}) and (\ref{eq-3}), we obtain 
	\begin{eqnarray*}
		\sum_{j\in\mathbb{Z}^2}g_{R,j}=T_{a}\left(\sum_{j\in\mathbb{Z}^2}\chi_{j}g_{R}\right) =T_{a}(g_{R})
	\end{eqnarray*}
	and
	\begin{eqnarray*}
		\operatorname{supp}\mathcal{F}_x(g_{R,j})\subset\left\{\xi\in\mathbb{R}^2:\frac{R}{4}\leq|\xi|\leq4R\right\}.
	\end{eqnarray*}
	Also, we give
	\begin{eqnarray*}
		\sum_j\|g_{R,j}\|_2^2\lesssim\sum_j\|\chi_jg_R\|_2^2\lesssim\|g_R\|_2^2.
	\end{eqnarray*}
	Let $\|j-\ell\|_\infty=\max_{1\le k\le 2}|j_k-\ell_k|.$ For each cube $Q_\ell$,  write
	\begin{eqnarray*}
		g_R=\sum_{\|j-\ell\|_\infty\leq200}g_{R,j}+\sum_{\|j-\ell\|_\infty>200}g_{R,j}.
	\end{eqnarray*}

	Let $x_\ell$ be the center of $Q_\ell$. Then,  \eqref{eq--38} gives
	\begin{eqnarray*}
		\left\|V^r_{[0,1]}e^{it\Delta}\left\lbrace\sum_{\|j-\ell\|_\infty\leq 200}g_{R,j}\right\rbrace \right\|_{L^p(Q_\ell)}&\lesssim&	\left\|V^r_{[0,1]}e^{it\Delta}\left\lbrace  \sum_{\|j-\ell\|_\infty\leq 200}g_{R,j}\right\rbrace  \right\|_{L^p(B^2(x_\ell,R))}\\
		&\lesssim&R^{1-2/p+2\beta+\varepsilon}\left\|\sum_{\|j-\ell\|_\infty\leq 200}g_{R,j}\right\|_2.
	\end{eqnarray*}
 For fixed $j$, we have  $\#\left\{\ell:\|j-\ell\|_\infty\leq 200
\right\}<\infty$. Then we interchange the order of summation,
	\begin{eqnarray*}
		\sum_\ell\left\|\sum_{\|j-\ell\|_\infty\leq200}g_{R,j}\right\|_2^2\lesssim\sum_\ell\sum_{\|j-\ell\|_\infty\le200}\|g_{R,j}\|_2^2.\lesssim\sum_j\|g_{R,j}\|_2^2\#\left\{\ell:\|j-\ell\|_\infty\leq200\right\}\lesssim\|g_R\|_2^2.
	\end{eqnarray*}
	By $p\geq2$, we obtain
	\begin{eqnarray*}
		\left(\sum_\ell\left\|V^r_{[0,1]}\left\lbrace e^{it\Delta}\sum_{\|j-\ell\|_\infty\leq 200}g_{R,j}\right\rbrace \right\|_{L^p(Q_\ell)}^p\right)^{1/p}\lesssim_{p,r,\varepsilon}R^{1-2/p+2\beta+\varepsilon}\|g_R\|_2.
	\end{eqnarray*}

	We now consider the far part. For $\nu=0,1$,
	\begin{eqnarray*}
		\partial_t^\nu e^{it\Delta}g_{R,j}(x)=\frac{1}{(2\pi)^2}\int_{\mathbb{R}^2}\int_{\mathbb{R}^2}(-i|\xi|^2)^\nu e^{i((x-y)\cdot\xi-t|\xi|^2)}a(\xi/R)\chi_j(y)g_R(y)dyd\xi.
	\end{eqnarray*}
	After the change of variables $\xi=R\eta$, the oscillatory phase is
	\begin{eqnarray*}
		R(x-y)\cdot\eta-R^2t|\eta|^2.
	\end{eqnarray*}
	If $x\in Q_\ell$, $y\in\operatorname{supp}\chi_j$,
	$\|j-\ell\|_\infty>200$, and $0\leq t\leq1$, then
	\begin{eqnarray*}
		\left|\nabla_\eta\left(R(x-y)\cdot\eta-R^2t|\eta|^2\right)\right|\gtrsim R^2(1+\|j-\ell\|_\infty).
	\end{eqnarray*}
	Repeated integration by parts in $\eta$ gives, for every $M\geq1$,
	\begin{eqnarray*}
		\sup_{\substack{x\in Q_\ell\\0\leq t\leq1}}\left|\partial_t^\nu e^{it\Delta}g_{R,j}(x)\right|\lesssim_{M,\nu}R^{3+2\nu-2M}(1+\|j-\ell\|_\infty)^{-M}\|\chi_jg_R\|_2.
	\end{eqnarray*}
	Here we have used
	\begin{eqnarray*}
		\|\chi_jg_R\|_1\lesssim R\|\chi_jg_R\|_2.
	\end{eqnarray*}
	It follows that
	\begin{eqnarray*}
		\sup_{x\in Q_\ell}V^r_{[0,1]}\{e^{it\Delta}g_{R,j}(x)\}\lesssim_MR^{5-2M}(1+\|j-\ell\|_\infty)^{-M}\|\chi_jg_R\|_2.
	\end{eqnarray*}
	Since $|Q_\ell|=R^2$,
	\begin{eqnarray*}
		\left\|V^r_{[0,1]}\{e^{it\Delta}g_{R,j}\}\right\|_{L^p(Q_\ell)}\lesssim_MR^{5+2/p-2M}(1+\|j-\ell\|_\infty)^{-M}\|\chi_jg_R\|_2.
	\end{eqnarray*}
	Discrete Young's inequality, followed by $\ell^2\subset\ell^p$, gives
	\begin{eqnarray*}
		\left(\sum_\ell\left\|V^r_{[0,1]}\left\lbrace e^{it\Delta}\sum_{\|j-\ell\|_\infty>200}g_{R,j}\right\rbrace  \right\|_{L^p(Q_\ell)}^p\right)^{1/p}\lesssim R^{5+2/p-2M}\|g_R\|_2.
	\end{eqnarray*}
	Since $M$ is arbitrary, the right-hand side is bounded by $R^{-100}\|g_R\|_2$ after choosing $M$ sufficiently large. Combining the near and far estimates, we obtain
	\begin{eqnarray}\label{eq--39}
		\left\|V^r_{[0,1]}\{e^{it\Delta}g_R\}\right\|_{L^p(\mathbb{R}^2)}\lesssim_{p,r,\varepsilon}R^{1-2/p+2\beta+\varepsilon}\|g_R\|_2.
	\end{eqnarray}
	Choose a smooth inhomogeneous Littlewood--Paley decomposition $f=g_{0}+\sum_{j>0}g_j,$ where $g_{j}=g_{R}|_{R=2^{j}}$. Set $s_0:=1-\frac2p+2\beta$. By subadditivity of the variation and
	\eqref{eq--39}, we have
	\begin{eqnarray*}
		\left\|V^r_{[0,1]}\{e^{it\Delta}\sum_{j>0}g_j\}\right\|_{L^p}\lesssim\sum_{j>0}2^{js_0+\varepsilon j}\|g_j\|_2.
	\end{eqnarray*}
	Let $s>s_0$, and choose $\varepsilon>0$ so that $s-s_0-\varepsilon>0$. Then
	\begin{eqnarray*}
		\sum_{j>0}2^{js_0+j\varepsilon}\|g_j\|_2\leq\left(\sum_{j>0}2^{-2j(s-s_0-\varepsilon)}\right)^{1/2}\left(\sum_{j>0}2^{2js}\|g_j\|_2^2\right)^{1/2}\lesssim\|f\|_{H^s(\mathbb{R}^2)}.
	\end{eqnarray*}
	The low-frequency contribution satisfies
	\begin{eqnarray*}
		\left\|V^r_{[0,1]}\{e^{it\Delta}g_{\leq1}\}\right\|_{L^p}\lesssim_p\|g_{\leq1}\|_2\lesssim\|f\|_{H^s}.
	\end{eqnarray*}
	This proves \eqref{eq--37}.
\end{proof}

\begin{theorem}\label{thm:global-positive-estimates}
	Let $2\leq p\leq\infty$, $1\leq r\leq\infty$. Then, for every $s> s(p,r)$,
	$$\left\|V^r_{[0,1]}\{e^{it\Delta}f\}\right\|_{L^p(\mathbb{R}^2)}\lesssim_{p,r,s}\|f\|_{H^s(\mathbb{R}^2)},$$
	where
	$$s(p,r)=
	\left\{\begin{array}{ll}
		\displaystyle \frac{2}{r}+\frac{1}{p}-\frac{1}{2},
		&1\leq r\leq2,\quad 2\leq p\leq6,\\[7pt]
		\displaystyle \frac{2}{r}-\frac{2}{p},
		&1\leq r\leq2,\quad 6\leq p\leq\infty,\\[7pt]
		\displaystyle \max\left\{m(p),\frac1p+\frac1r\right\},
		&2<r\leq\infty
	\end{array}\right.$$ 
    and
	$$m(p)=
		\left\{\begin{array}{ll}
		\displaystyle \frac{2}{p}-\frac{1}{3},
		&2\leq p\leq3,\\[5pt]
		\displaystyle 1-\frac{2}{p},
		&3\leq p\leq\infty.
	\end{array}\right.$$
\end{theorem}

\begin{proof}
	When $r=\infty$, the assertion is the corresponding Schr\"odinger maximal estimate under the convention $V^\infty_I F=2\sup_{t\in I}|F(x,t)|$. We therefore assume $1\leq r<\infty$ below.

	We combine the local estimates of Section 3 with Theorem \ref{Th-3}.  In the notation of that theorem, a local bound with the factor $R^{\beta+\varepsilon}$ gives the global threshold
	$$s_0=1-\frac{2}{p}+2\beta.$$
    Suppose first that $1\leq r\leq2$. If $2\leq p\leq6$, then Lemma \ref{Le-5} gives
	$$\beta=\frac1r-\frac34+\frac{3}{2p}.$$
	Consequently,
	\begin{eqnarray*}
		s_0=1-\frac{2}{p}+2\left(\frac{1}{r}-\frac{3}{4}+\frac{3}{2p}\right)=\frac{2}{r}+\frac{1}{p}-\frac{1}{2}.
	\end{eqnarray*}
	In this range the last expression is the largest of the five quantities in the definition of $ s(p,r)$.  If $6\leq p\leq\infty$, the same corollary instead gives
	$$\beta=\frac{1}{r}-\frac{1}{2},$$
	and hence
	\begin{eqnarray*}
		s_0=1-\frac{2}{p}+2\left(\frac{1}{r}-\frac{1}{2}\right)=\frac{2}{r}-\frac{2}{p}.
	\end{eqnarray*}
	This is again $ s(p,r)$.  Notice that these calculations include
	$r=1$, $r=2$, $p=2$, and $p=6$.
	We next take $2\leq p<3$ and put
	$$q_p=\frac{3p}{3-p}.$$
	For $r\geq q_p$, Lemma \ref{Le-2} yields
	$$\beta=\frac{2}{p}-\frac{2}{3}.$$
	Therefore
	\begin{eqnarray*}
		s_0=1-\frac{2}{p}+2\left(\frac{2}{p}-\frac{2}{3}\right)=\frac{2}{p}-\frac{1}{3}=m(p).
	\end{eqnarray*}
	Since $r\geq q_p$ implies
	$$\frac{1}{p}+\frac{1}{r}\leq\frac{1}{p}+\frac{1}{q_p}=\frac{2}{p}-\frac{1}{3},$$
	we have $s_0= s(p,r)$ in this range.  For
	$2<r\leq q_p$, Lemma \ref{Le-4} gives
	$$\beta=-\frac{1}{2}+\frac{3}{2p}+\frac{1}{2r}.$$
	It follows that
	\begin{eqnarray*}
		s_0=1-\frac{2}{p}+2\left(-\frac12+\frac{3}{2p}+\frac{1}{2r}\right)=\frac{1}{p}+\frac{1}{r}.
	\end{eqnarray*}
	Here $r\leq q_p$ gives
	$1/p+1/r\geq2/p-1/3=m(p)$, so this value is also $ s(p,r)$. At $p=3$ and $r>2$, the first estimate in Lemma \ref{Le-6} has
	$$\beta=\frac{1}{2r}.$$
	Thus
	$$s_0=1-\frac{2}{3}+\frac{1}{r}=\frac{1}{3}+\frac{1}{r}= s(3,r).$$
	Let $3<p<6$ and $r>2$. Lemma \ref{Le-7} gives
	$$\beta=\max\left\{0,-\frac12+\frac{3}{2p}+\frac{1}{2r}\right\}.$$
	Consequently,
	\begin{eqnarray*}
		s_0=1-\frac{2}{p}+2\max\left\{0,-\frac{1}{2}+\frac{3}{2p}+\frac{1}{2r}\right\}=\max\left\{1-\frac{2}{p},\frac{1}{p}+\frac{1}{r}\right\}=s(p,r).
	\end{eqnarray*}
	The equality of the two branches occurs at $r=p/(p-3)$.

	Finally, let $6\leq p\leq\infty$ and $r>2$.  The second estimate in Lemma \ref{Le-6} gives $\beta=0$, and therefore
	$$s_0=1-\frac{2}{p}.$$
	Since
	$$\frac{1}{p}+\frac{1}{r}<\frac{1}{p}+\frac{1}{2}\leq1-\frac{2}{p},$$
	the last quantity equals $ s(p,r)$.  All admissible pairs $(p,r)$ have now been covered. Applying Theorem \ref{Th-3} in each case completes the proof of the sufficiency part of Theorem 1.1. 
\end{proof}

\section{The proof of the necessity part in Theorem \ref{Th-1}}
\subsection{Counterexample in the range $p<2$ or $r<1$}
The restriction $r\geq1$ is already necessary for a single
Littlewood--Paley piece.  Choose a nonzero function
$f\in\mathcal{S}(\mathbb{R}^2)$ satisfying
$$\operatorname{supp}\mathcal{F}_{x}(f)\subset\left\{\xi\in\mathbb{R}^2:\frac{1}{2}\leq|\xi|\leq2\right\},\quad \text{and} \quad\Delta f\not\equiv0.$$
and set
$$u(x,t):=e^{it\Delta}f(x).$$
Moreover, $\partial_tu(x,0)=i\Delta f(x).$ Hence there exist a compact set $E\subset\mathbb{R}^2$ of positive measure and a constant $c_0>0$ such that
$$|\Delta f(x)|\geq2c_0,\qquad x\in E.$$
Since $\partial_tu$ is continuous in $(x,t)$, there exists
$0<\delta\leq1$ such that
$$\left|\partial_t u(x,t)-i\Delta f(x)\right| \leq c_0$$
for every $x\in E$ and $0\leq t\leq\delta$.

For each positive integer $R$, let
$$t_j=\frac{j\delta}{R},\qquad j=0,\ldots,R.$$
For $x\in E$, the fundamental theorem of calculus gives
$$\begin{aligned}
	u(x,t_j)-u(x,t_{j-1})=\frac{\delta}{R}i\Delta f(x)+\int_{t_{j-1}}^{t_j}\left(\partial_t u(x,t)-i\Delta f(x)\right)dt.
\end{aligned}
$$
Consequently,
$$|u(x,t_j)-u(x,t_{j-1})|\geq|\frac{\delta}{R}i\Delta f(x)| -\left| 
\int_{t_{j-1}}^{t_j}\left(\partial_t u(x,t)-i\Delta f(x)\right)dt\right|\geq c_0\frac{\delta}{R}.$$
It follows that
$$V_{[0,1]}^r\{e^{it\Delta}f(x)\}\geq\left(\sum_{j=1}^R|u(x,t_j)-u(x,t_{j-1})|^r\right)^{1/r} \geq c_0\delta R^{1/r-1}.$$
If $0<r<1$, then $1/r-1>0$.  Let $R\to\infty$, we obtain
$$V_{[0,1]}^r\{e^{it\Delta}f (x)\}=\infty,\qquad x\in E.$$
Therefore no $L^p$ estimate for the homogeneous $r$-variation can
hold when $0<r<1$. The restriction $p\geq2$ follows from a theorem of H\"ormander on translation-invariant operators. Although the variation operator itself is not linear, we can construct a linear operator whose absolute value is pointwise dominated by the variation operator.  Suppose that
$$\bigl\|V^r_{[0,1]}\{e^{it\Delta}f\}\bigr\|_{L^p(\mathbb{R}^2)}\lesssim\|f\|_{H^s(\mathbb{R}^2)} $$
with $f\in \mathcal{S}(R^{2})$. Let $g$ be a nonzero Littlewood--Paley function is supported in $\left\{\xi\in\mathbb{R}^2:\frac12\leq|\xi|\leq1\right\},$ and fix $t_0\in(0,1]$. Let $T_{g}(f):=\mathcal{F}_{x}^{-1}(g\mathcal{F}_{x}(f))$.  Then we have 
$$
\begin{aligned}
	\bigl\|(e^{it_0\Delta}-I)T_{g}(f)\bigr\|_{L^p(\mathbb{R}^2)}
	&\leq
	\bigl\|
	 V^r_{[0,1]}\{e^{it\Delta} T_{g}(f)\}
	\bigr\|_{L^p(\mathbb{R}^2)}
	\lesssim
	\|T_{g}(f)\|_{H^s(\mathbb{R}^2)}\lesssim \|f\|_{H^s(\mathbb{R}^2)}.
\end{aligned}
$$
Furthermore, $T=\left( e^{it_0\Delta}-I\right)T_{g}$ would define a bounded translation-invariant operator from $L^2(\mathbb{R}^2)$ to $L^p(\mathbb{R}^2)$. Note that $\bigl(e^{-it_0|\xi|^2}-1\bigr)g(\xi)\not\equiv0$ ensures that $T_{g}$ is nonzero.  H\"ormander's theorem for translation invariant operators then implies $p\geq2$; see \cite[Theorem 1.1]{H1960}.  Hence the variation estimate cannot hold when $p<2$.

\subsection{Counterexample in the range  $s< m(p)$}
\begin{theorem}
	Let $2\leq p\leq\infty$ and $1\leq r<\infty$.  If
$$
\bigl\|V^r_{[0,1]}\{e^{it\Delta}f \}\bigr\|_{L^p(\mathbb{R}^2)}
\lesssim \|f\|_{H^s(\mathbb{R}^2)}
$$
holds for all $f\in H^s(\mathbb{R}^2)$, then necessarily
$$
s\geq
\begin{cases}
	\displaystyle \frac{2}{p}-\frac13,
	& 2\leq p\leq3,\\[4pt]
	\displaystyle 1-\frac{2}{p},
	& 3\leq p\leq\infty.
\end{cases}
$$
\end{theorem}

\begin{proof}
Let $2\leq p<3$. Bourgain's construction \cite[Section 2]{B2016}, applied at frequency $R^2$, provides a function $f\in\mathcal S(\mathbb{R}^2)$ satisfying
$$\|f\|_2=1,\qquad\operatorname{supp}\mathcal F_x f\subset\{\xi:|\xi|\approx R^2\},$$
and a measurable set $E\subset B(0,1)$, with $|E|\gtrsim 1$, such that
$$\sup_{0<t<cR^{-2}}|e^{it\Delta}f(x)|\gtrsim R^{2/3},\qquad x\in E,$$
where $0<c<1$ is independent of $R$; see also \cite[Proposition 2.1]{DKWZ2020}. Set $g(x)=R^{-1}f(x/R)$. Then
$$\|g\|_2=1,\qquad\operatorname{supp}\mathcal F_x g\subset\{\xi:|\xi|\approx R\},\qquad\|g\|_{H^s}\approx R^s.$$
By the scaling identity
$$e^{it\Delta}g(x)=R^{-1}\bigl(e^{i(t/R^2)\Delta}f\bigr)(x/R),$$
we have
$$\sup_{0<t<1}|e^{it\Delta}g(x)|\gtrsim R^{-1/3},\qquad x\in RE.$$
Since $|RE|\gtrsim R^2$, this gives
$$\left\|\sup_{0<t<1}|e^{it\Delta}g|\right\|_{L^p(\mathbb{R}^2)}\gtrsim R^{2/p-1/3}.$$
Take the two time points $0$ and $t$ in the definition of the variation,  we obtain
$$\bigl\|V^r_{[0,1]}\{e^{it\Delta}g\}\bigr\|_{L^p(\mathbb{R}^2)}\geq\left\|\sup_{0<t<1}|e^{it\Delta}g|\right\|_{L^p(\mathbb{R}^2)}-\|g\|_{L^p(\mathbb{R}^2)}\gtrsim R^{2/p-1/3}-R^{1-2/p}.$$
As $p<3$, we have $2/p-1/3>1-2/p$. Hence, for all sufficiently large $R$,
$$\bigl\|V^r_{[0,1]}e^{it\Delta}g\bigr\|_{L^p(\mathbb{R}^2)} \gtrsim R^{2/p-1/3}.$$
Consequently, the asserted $H^s$ estimate requires $s\geq2/p-1/3$.

	For $3\leq p\leq\infty$, we use the standard frequency concentration example underlying the necessary condition $s\geq1-2/p$ for the global Schr\"odinger maximal estimate; see \cite[Theorem 9]{RV2007}. For variation, it suffices to consider a single time increment.
	
	Choose a fixed nonzero function	$\phi\in\mathcal S(\mathbb{R}^2)$ such that
	$$\operatorname{supp}\mathcal F_x\phi\subset\{\xi\in\mathbb{R}^2:1/2\leq|\xi|\leq2\},$$
	and set	$f_R(x)=R\phi(Rx)$.
    Note that 
	$$|e^{-i|\xi|^2}-1|=2\left|\sin\left(\frac{|\xi|^2}{2}\right)\right|\geq 2\sin(1/8)>0\qquad\text{when }1/2\leq|\xi|\leq2,$$
	which indicates $(e^{i\Delta}-I)\phi\not\equiv0$. Take the two time points $0$ and $R^{-2}$ gives	
	$$V^r_{[0,1]}\{e^{it\Delta}f_R(x)\}\geq|e^{iR^{-2}\Delta}f_R(x)-f_R(x)|.$$
	Therefore,
	\begin{eqnarray*}
		\bigl\|V^r_{[0,1]}\{e^{it\Delta}f_R\}\bigr\|_{L^p}\geq\left(\int_{\mathbb{R}^{2}}\left|Re^{i\Delta}\phi(Rx)-R\phi(Rx)\right|^{p}dx\right)^{1/p}=R^{1-2/p}\|(e^{i\Delta}-I)\phi\|_{L^{p}(\mathbb{R}^{2})}
	\end{eqnarray*}
	Since 	$\|f_R\|_2=\|\phi\|_2,\|f_R\|_{H^s}\approx R^s,$the assumed estimate forces
	$$R^{1-2/p}\lesssim R^s.$$
	Let $R$ tend to infinity proves $s\geq1-2/p$. This also gives $s\geq1/3$ at $p=3$. 
\end{proof}

\subsection{Counterexample in the range   $s< \frac{1}{p}+\frac{1}{r}$}
In this section, we adapt the high-frequency Galilean wave-packet construction used by Lee, Rogers, and Seeger to the $r$-variation of the Schrödinger equation \cite{LRS2013}, obtaining a necessary condition for the corresponding variation estimate.
\begin{theorem}
	Let $2\leq r<\infty$ and $2\leq  p<\infty$.  If
	$$\bigl\|V^r_{[0,1]}\{e^{it\Delta}f\}\bigr\|_{L^p(\mathbb{R}^2)}\lesssim \|f\|_{H^s(\mathbb{R}^2)}$$
	holds for all $f\in H^s(\mathbb{R}^2)$, then necessarily
	$$s\ge \frac{1}{p}+\frac{1}{r}.$$
\end{theorem}
\begin{proof}
	Let $ \mathcal{F}_x(\varphi)\in C^{\infty}_{c}$ satisfy $\mathcal{F}_x(\varphi)\geq 0$ and 
	$$  \mathcal{F}_x(\varphi)=\left\{
	\begin{array}{rcl}
		&1,\quad\quad \quad \text{ for } \xi \in B^{2}(0,1/4) ;\\
		&O(R^{-1000}),\text{ for } \xi \nin B^{2}(0,1/2).
	\end{array}
	\right.$$
	 Define
	$$ f_{R}(x):= e^{iRx_{1}}\varphi(x),$$
	where $R\geq 1$. By a directly computation, we obtain
	$$\mathcal{F}_x(f_{R})(\xi)=\int_{\mathbb{R}^{2}}e^{-ix\xi}e^{iRx_{1}}\varphi(x)dx=\mathcal{F}_x(\varphi)(\xi-Re_{1}),$$
	where $e_{1}:=(1,0)$, which indicates $supp \mathcal{F}_x(f_{R})\subset B^{2}(0,2R)$. Then we have
	\begin{eqnarray*}
		e^{it\Delta} f_{R}(x)&=&\int_{\mathbb{R}^{2}}e^{ix\xi-it|\xi|^{2}}\mathcal{F}_x(\varphi)(\xi-Re_{1})d\xi\\
     	&=& \int_{\mathbb{R}^{2}} e^{ix(\xi+Re_{1})}e^{-it|\xi+Re_{1}|^{2}}\mathcal{F}_x(\varphi)(\xi)d\xi	\\
		&=&e^{ixRe_{1}-itR^{2}}\int_{\mathbb{R}^{2}}e^{i(x-2tRe_{1})\xi}e^{-it|\xi|^{2}}\mathcal{F}_x(\varphi)(\xi)d\xi\\
		&=&e^{iRx_{1}-itR^{2}}e^{it\Delta} \varphi(x-2Rte_{1}).
	\end{eqnarray*}
	Set
	$$S_R:=\bigl\{x\in\mathbb{ R}^2:R/2\leq x_1\leq 3R/2,\ |x_2|\le 1/4\bigr\},\qquad|S_R|=R/2.$$
	Then
   $$|S_R|	=(3R/2-R/2)/2=R/2.$$
	Fix $x\in S_{R}$, and set
	$$t_x:=\frac{x_1}{2R}.$$
	Since $R/2\leq x_1\leq 3R/2$, we have $1/4\leq t_x\leq 3/4.$

	For sufficiently large $R\geq 2$, the interval
	$$I_x:=\left[t_x-\frac{1}{4R},t_x+\frac{1}{4R}\right]$$
	is contained in $(0,1)$. Put
	$$J_R:=\left\lfloor\frac{R}{2\pi}\right\rfloor,\qquad t_j:=t_x-\frac{ 1}{4R}+\frac{j\pi}{R^2},\qquad 0\leq j\leq J_R.$$
	Indeed,
	$$t_0=t_x-\frac{1}{4R},\qquad t_{J_R}\leq t_x-\frac{1}{4R}+\frac{R}{2\pi}\frac{\pi}{R^2}=t_x+\frac{1}{4R},$$
	so $t_j\in I_x\subset(0,1)$. Moreover,
	$$\frac{R}{2\pi}-1\leq J_R\leq\frac{R}{2\pi},\qquad\text{hence}\qquad J_R\approx R.$$

	For every $0\leq j\leq J_R$,
	\begin{eqnarray*}
		x_1-2Rt_j&=&x_1-2R\left(\frac{x_1}{2R}-\frac{1}{4R}+\frac{j\pi}{R^2}\right)\\
		&=&\frac{1}{2}-\frac{2j\pi}{R}.
	\end{eqnarray*}
	Since $0\leq j\leq J_R\leq  R/(2\pi)$,
	\begin{eqnarray*}
	|x-2Rt_je_1|=\left(	|x_1-2Rt_j|^2 +|x_{2}|^{2}\right)^{1/2}<1.
	\end{eqnarray*}
	Note that
	\begin{eqnarray}\label{eq--40}
		\left|(x-2Rt_je_1)\cdot\xi-t_j|\xi|^2\right|\leq |\xi|+|\xi|^2\leq\frac{\pi}{3}
	\end{eqnarray}
By \eqref{eq--40} and $0<t_j<1$, we have
	\begin{eqnarray*}
		\text{Re}e^{it_{j}\Delta}\varphi(x-2Rt_je_1)&\geq&\int_{|\xi|\leq1/2}\cos\left((x-2Rt_je_1)\cdot\xi-t_j|\xi|^2\right)\mathcal{F}_x(\varphi)(\xi) d\xi\\
		 &\geq& \pi/32+O(R^{-1000}).
	\end{eqnarray*}
	The Galilean identity yields
	$$e^{it_{j}\Delta}f_R(x)=e^{i(Rx_1-R^2t_j)}e^{it_{j}\Delta}\varphi(x-2Rt_{j}e_{1}).$$
	Furthermore,
	$$R^2(t_{j+1}-t_j)=\pi,\qquad e^{-iR^2t_{j+1}}=-e^{-iR^2t_j},$$
	and therefore
	\begin{eqnarray*}
		\bigl|e^{it_{j+1}\Delta}f_R(x)-e^{it_{j}\Delta}f_R(x)\bigr|&=&|e^{it_{j+1}\Delta}\varphi(x-2Rt_{j+1}e_{1}) +e^{it_{j}\Delta}\varphi(x-2Rt_{j}e_{1})|\\
		&\geq&\text{Re}e^{it_{j+1}\Delta}\varphi(x-2Rt_{j+1}e_{1})
		+\operatorname{Re}e^{it_{j}\Delta}\varphi(x-2Rt_{j+1}e_{1})\\
		&\geq& \pi/16.
	\end{eqnarray*}
	For every $x \in S_{R}$,
	\begin{eqnarray*}
		V^r_{[0,1]}\{e^{it\Delta}f_R(x)\}&\geq&\left(\sum_{j=0}^{J_R-1}\bigl|e^{it_{j+1}\Delta}f_R(x)-e^{it_{j}\Delta} f_R(x)\bigr|^r\right)^{1/r}\geq \pi J_R^{1/r}/16\gtrsim R^{1/r}.
	\end{eqnarray*}
which indicates
	\begin{eqnarray}\label{eq--41}
		\big\|V^r_{[0,1]}\{e^{it\Delta}f_R\}\big\|_{L^p(\mathbb{R}^2)}&\geq&\left(\int_{S_R}\big|V^r_{[0,1]}\{e^{it \Delta} f_R(x)\}\big |^pdx\right) ^{1/p} \gtrsim  R^{1/p+1/r}. 
	\end{eqnarray}
	Suppose that
	$$\bigl\|V^r_{[0,1]}\{e^{it \Delta}f_R\}\bigr\|_{L^p(\mathbb{R}^2)}\lesssim\|f_{R}\|_{H^s(\mathbb{R}^2)}.$$
	 By \eqref{eq--41}, we have
	$$R^{1/p+1/r}\lesssim\bigl\|V^r_{[0,1]}\{e^{it\Delta}f_R\}\bigr\|_{L^p(\mathbb{R}^2)}\lesssim\|f_R\|_{H^s(\mathbb{R}^2)}\approx R^s.$$
	Since this holds for arbitrarily large $R$, we have $s\ge \frac{1}{p}+\frac{1}{r}$ is necessary.
\end{proof}
\subsection{Counterexample in the range $s< \frac{1}{p}+\frac{2}{r}-\frac{1}{2}$}

In this subsection, we construct a counterexample which yields the necessary condition
\begin{eqnarray*}
	s\geq \frac{1}{p}+\frac{2}{r}-\frac{1}{2}.
\end{eqnarray*}
In the range $1\leq r\leq 2$ and $2\leq p\leq 6$, this gives the first necessary-condition branch in the phase diagram.

\begin{theorem}
Let $1\leq r\leq 2$ and $2\leq p\leq 6$.  If
$$\left\|V^r_{[0,1]}\{e^{it \Delta}f\}\right\|_{L^p(\mathbb{R}^2)}\leq C\|f\|_{H^s(\mathbb{R}^2)}$$
	holds for all $f\in H^s(\mathbb{R}^2)$, then necessarily
	$$s\geq \frac{1}{p}+\frac{2}{r}-\frac{1}{2}.$$
\end{theorem}
\begin{proof}
	Choose a nonzero function $a\in C_c^\infty(\mathbb{R}^2)$ such that
	\begin{eqnarray*}
		a\geq0,\quad\int_{\mathbb{R}^{2}}a(\xi)d\xi>0,\quad\operatorname{supp}a(\xi)\subset\left\{\eta\in\mathbb{R}^2:|\xi_1|\leq\frac{1}{4},\|\xi_2|\leq\frac{1}{4}\right\}.
	\end{eqnarray*}
 For $R\gg 1$, define $f_R$, by
	\begin{eqnarray*}
		\mathcal{F}_x(f_R)(\xi_1,\xi_2):=R^{1/2}a\bigl(R(\xi_1-R),\xi_2\bigr),
	\end{eqnarray*}
which implies the Fourier support of $f_R$ is contained in the rectangle
	\begin{eqnarray*}
		\operatorname{supp}\mathcal{F}_x(f_R)\subset\left\{(\xi_1,\xi_2):|\xi_1-R|\leq\frac{1}{4R},\quad|\xi_2|\leq\frac{1}{4}\right\}.
	\end{eqnarray*}
	In particular, $|\xi|\approx R\text{ on }\operatorname{supp}\mathcal{F}_x(f_R)$.
	Making the change of variables	and using Plancherel's theorem, we obtain
	\begin{eqnarray*}
		\|f_R\|_2^2&\approx&\int_{\mathbb{R}^2}R\left|a\bigl(R(\xi_1-R),\xi_2\bigr)\right|^2
	    d\xi\\
		&=&\int_{\mathbb{R}^2}|a(\eta_1,\eta_2)|^2d\eta\\
		&=&\|a\|_2^2.
	\end{eqnarray*}
	Therefore, $\|f_R\|_2\approx 1$. Since $|\xi|\approx R$ on the Fourier support of $f_R$, we also have
	\begin{eqnarray*}
		\|f_R\|_{H^s(\mathbb{R}^2)}\approx R^s\|f_R\|_2 \approx R^s.
	\end{eqnarray*}

Let	$\xi_1=R+\frac{\eta_1}{R}, \xi_2=\eta_2,$ we obtain
	\begin{eqnarray}\label{eq--43}
		x_1\xi_1+x_2\xi_2-t(\xi_1^2+\xi_2^2)=Rx_1-R^2t+\eta_1\left(\frac{x_1}{R}-2t\right)+x_2\eta_2-t\eta_2^2-\frac{t\eta_1^2}{R^2}.
	\end{eqnarray}
		We next calculate $e^{it\Delta}f_R$. By\eqref{eq--43}, we obtain
	\begin{eqnarray}\label{eq--44}
		e^{it\Delta}f_R(x)&=&R^{1/2}\int_{\mathbb{R}^2}e^{i(x_1\xi_1+x_2\xi_2-t(\xi_1^2+\xi_2^2))}a\bigl(R(\xi_1-R),\xi_2\bigr)d\xi\\
		&=&R^{-1/2}e^{i(Rx_1-R^2t)}A_R(x,t),\notag
	\end{eqnarray}
	where
	\begin{eqnarray*}
		A_R(x,t)=\int_{\mathbb{R}^2}a(\eta_1,\eta_2)\exp\left(i\Phi_R(x,t,\eta)\right)d\eta
	\end{eqnarray*}
	and
	\begin{eqnarray*}
		\Phi_R(x,t,\eta)=\eta_1\left(\frac{x_1}{R}-2t\right)+x_2\eta_2-t\eta_2^2-\frac{t\eta_1^2}{R^2}.
	\end{eqnarray*}

	Consider the spatial region
	\begin{eqnarray*}
		\Omega_R=\left\{x\in\mathbb{R}^2:\frac{R}{2}\leq x_1\leq\frac{3R}{2},\quad|x_2|\leq 1\right\}
	\end{eqnarray*}
	and the time interval
	\begin{eqnarray*}
		I=\left[\frac{1}{4},\frac{1}{2}\right].
	\end{eqnarray*}
	If $x\in\Omega_R$ and $t\in I$, then $\left|\frac{x_1}{R}-2t\right|\leq 1.$ For $\eta\in\operatorname{supp}a$, it follows that
	\begin{eqnarray*}
		|\Phi_R(x,t,\eta)|&\leq&|\eta_1|\left|\frac{x_1}{R}-2t\right|+|x_2||\eta_2|+t|\eta_2|^2+\frac{t|\eta_1|^2}{R^2}\\
		&\leq&\frac{1}{2}+\frac{1}{16}=\frac{9}{16}\leq \frac{\pi}{3}.
	\end{eqnarray*}
	Thus, $\cos\Phi_R(x,t,\eta)\geq\frac{1}{2}$, whenever $x\in\Omega_R$, $t\in I$, and $\eta\in\operatorname{supp}a$. Since $\int_{\mathbb{R}^{2}}a>0$, we conclude that
	\begin{eqnarray*}
		\operatorname{Re}A_R(x,t)&=&\int_{\mathbb{R}^2}a(\eta)\cos\Phi_R(x,t,\eta)d\eta\\
		&\geq&\frac{1}{2}\int_{\mathbb{R}^2}a(\eta)d\eta>0
	\end{eqnarray*}
	uniformly for $x\in\Omega_R$ and $t\in I$.

	We now exploit the rapid oscillation of the
	$e^{-iR^2t}$. Let
	\begin{eqnarray*}
		t_j=\frac14+\frac{j\pi}{R^2},\qquad j=0,1,\ldots,J_R,
	\end{eqnarray*}
	where $J_R=\left\lfloor\frac{R^2}{4\pi}\right\rfloor. $ 	Then
	\begin{eqnarray*}
		\frac{1}{4}\leq t_0<t_1<\cdots<t_{J_R}\leq\frac{1}{2},\qquad J_R\approx R^2,
	\end{eqnarray*}
	and
	\begin{eqnarray}\label{eq--45}
		e^{-iR^2t_j}=e^{-iR^2/4}e^{-ij\pi}=e^{-iR^2/4}(-1)^j.
	\end{eqnarray}

	Hence, for every $1\leq j\leq J_R$,  By
	\eqref{eq--44} and \eqref{eq--45}, we have
	\begin{eqnarray*}
		&&\left|e^{it_j\Delta}f_R(x)-e^{it_{j-1}\Delta}f_R(x)\right|\\
		&=&R^{-1/2}\left|(-1)^jA_R(x,t_j)-(-1)^{j-1}A_R(x,t_{j-1})\right|\\
		&=&R^{-1/2}\left|A_R(x,t_j)+A_R(x,t_{j-1})\right|\\
		&\geq&R^{-1/2}\operatorname{Re}\left(A_R(x,t_j)+A_R(x,t_{j-1})\right)\\
		&\geq&2c_aR^{-1/2}.
	\end{eqnarray*}
	Using this particular time sequence in the definition of the
	$r$-variation gives
	\begin{eqnarray*}
		V^r_{[0,1]}\{e^{it\Delta}f_R(x)\}&\geq&\left(\sum_{j=1}^{J_R}\left|e^{it_j\Delta}f_R(x)-e^{it_{j-1}\Delta}f_R(x)\right|^r\right)^{1/r}\\
		&\geq&2c_aR^{-1/2}J_R^{1/r}\\
		&\gtrsim&R^{2/r-1/2}
	\end{eqnarray*}
	for every $x\in\Omega_R$. Since $|\Omega_R|\approx R$, we obtain
	\begin{eqnarray}\label{eq--46}
		\left\|V^r_{[0,1]}\{e^{it\Delta}f_R\}\right\|_{L^p(\mathbb{R}^2)}&\geq&\left\|V^r_{[0,1]}\{e^{it\Delta}f_R\}\right\|_{L^p(\Omega_R)}\notag\\
		&\gtrsim&R^{2/r-1/2}|\Omega_R|^{1/p}\notag\\
		&\gtrsim&R^{1/p+2/r-1/2}.
	\end{eqnarray}
Suppose that
\begin{eqnarray}\label{eq--47}
	\bigl\|V^r_{[0,1]}\{e^{it\Delta}f_R\}\bigr\|_{L^p(\mathbb{R}^2)}\lesssim\|f_{R}\|_{H^s(\mathbb{R}^2)}\lesssim R^{s}
\end{eqnarray}
	Combining \eqref{eq--46} and \eqref{eq--47}, we obtain
	\begin{eqnarray*}
		R^{1/p+2/r-1/2}\lesssim R^s.
	\end{eqnarray*}
	Let $R\rightarrow\infty$ yields
	\begin{eqnarray*}
		s\geq\frac{1}{p}+\frac{2}{r}-\frac{1}{2}.
	\end{eqnarray*}
	This proves the proposition.
\end{proof}

\subsection{Counterexample in the range $s<2/r-2/p$}

The following construction gives the necessary condition $s\ge 2/r-2/p$, which is the relevant lower bound in the range $1\leq r\leq 2$ and $p\geq 6$.
\begin{theorem}\label{th-5.4}
		Let $1\leq r\leq 2$ and $p\geq 6$.  If
	$$\left\|V^r_{[0,1]}\{e^{it\Delta}f  \}\right\|_{L^p(\mathbb{R}^2)}\lesssim\|f\|_{H^s(\mathbb{R}^2)}$$
	holds for all $f\in H^s(\mathbb{R}^2)$, then necessarily
	$$s\ge \frac{2}{r}-\frac{2}{p}.$$
\end{theorem}

\begin{proof}
	Choose a nonzero nonnegative function $a\in C_c^\infty((0,\frac{1}{100}))$ such that $\int_{\mathbb{R}}a >0 $. For every sufficiently large $R$, define $f_R$ by
	\begin{eqnarray}\label{eq--48}
		\mathcal{F}_x(f_R)(\xi)=c_R a\left(R(|\xi|-R)\right),\qquad\xi\in\mathbb{R}^2,
	\end{eqnarray}
	where the positive constant $c_R$ is chosen so that $\|f_R\|_2=1$. Since $a$ is supported in $(0,\frac{1}{100})$, we have
	$$\operatorname{supp}\mathcal{F}_x(f_R)\subset\left\{\xi\in\mathbb{R}^2:R\le|\xi|\le R+\frac{1}{100R}\right\}.$$
    By  the change of variables, we can obtain
	$$\rho=|\xi|,\qquad u=R(\rho-R),\qquad\rho=R+\frac{u}{R},\qquad d\rho=\frac{du}{R},$$
	we obtain
	\begin{eqnarray}\label{eq--49}
		\|\mathcal{F}_x(f_R)\|_2^2\approx c_R^2\int_R^{R+1/(100R)}\left|a\left(R(\rho-R)\right)\right|^2\rho d\rho=c_R^2\int_0^{1/100}|a(u)|^2\left(1+\frac{u}{R^2}\right)du\approx c_R^2.
	\end{eqnarray}
	Consequently, after choosing the Fourier normalization consistently, $c_R\approx1$ and $\|f_R\|_2=1$. By the Fourier representation,
	$$e^{it\Delta}f_R(x)=c_R\int_{\mathbb{R}^2}e^{i(x\cdot\xi-t|\xi|^2)}a\left(R(|\xi|-R)\right)d\xi.$$
	Factoring out the principal temporal oscillation $e^{-iR^2t}$ gives
	\begin{eqnarray}	\label{eq--50}
		e^{it\Delta}f_R(x)=e^{-iR^2t}A_R(x,t),
	\end{eqnarray}
	where
	$$A_R(x,t)=c_R\int_{\mathbb{R}^2}e^{i(x\cdot\xi-t(|\xi|^2-R^2))}a\left(R(|\xi|-R)\right)d\xi.$$
 Write
	$$|\xi|=R+\frac{u}{R},\qquad0\le u\le \frac{1}{100}.$$
	Then
	$$|\xi|^2-R^2=2u+\frac{u^2}{R^2},$$
	and hence, for all sufficiently large $R$,
	$$\left||\xi|^2-R^2\right|\le \frac{3}{100}.$$
    For
	$|x|\le \frac{1}{100R}$, then
	$$|x\cdot\xi|\le\frac{1}{100R}\left(R+\frac{1}{100R}\right)\le\frac{1}{50}.$$
	Therefore, for every $0\le t\le1$ and every
	$\xi\in\operatorname{supp}\mathcal{F}_x(f_R)$,
	$$\left|x\cdot\xi-t(|\xi|^2-R^2)\right|\le\frac{1}{20}.$$
		Indeed, using $a\ge0$, we have
	\begin{eqnarray*}
		\operatorname{Re}A_R(x,t)&\ge&c_R\cos\left(\frac{\pi}{6}\right)\int_{\mathbb{R}^2}a\left(R(|\xi|-R)\right)d\xi\gtrsim c_{0},
	\end{eqnarray*}
	which indicates
	\begin{eqnarray}\label{eq--51}
		\operatorname{Re}A_R(x,t)\ge c_0\qquad\text{whenever}|x|\le\frac{1}{100R}\quad\text{and}\quad0\le t\le1,
	\end{eqnarray}
	where $c_0>0$ is independent of $R$.

	We now select a sequence of times at which the principal carrier
	$e^{-iR^2t}$ alternates its sign. Let
	$$t_j=\frac{j\pi}{R^2},\qquad j=0,1,\ldots,J_R,\qquad J_R=\left\lfloor\frac{R^2}{\pi}\right\rfloor.$$
	Then $0\le t_j\le1$, $J_R\approx R^2$, and
	$$e^{-iR^2t_j}=e^{-ij\pi}=(-1)^j.$$
	Consequently, by \eqref{eq--50},
	\begin{eqnarray}\label{eq--52}
		\left|e^{it_{j+1}\Delta}f_R(x)-e^{it_j\Delta}f_R(x)\right|=\left|A_R(x,t_{j+1})+A_R(x,t_j)\right|\ge 2c_0
	\end{eqnarray}
	for every $j=0,\ldots,J_R-1$ and every $|x|\le c_1/R$. In the last inequality, we used \eqref{eq--51} and
	$$|z_1+z_2|\geq \operatorname{Re}z_1+\operatorname{Re}z_2$$
	when both real parts are positive.

	Take the above sequence of times in the definition of the $r$-variation and using \eqref{eq--52}, we obtain
	\begin{eqnarray}\label{eq--53}
		V^r_{[0,1]}\{e^{it\Delta}f_R(x)\}&\geq&\left(\sum_{j=0}^{J_R-1}\left|e^{it_{j+1}\Delta}f_R(x)-e^{it_j\Delta}f_R(x)\right|^r\right)^{1/r}\notag\\
		&\ge&2c_0J_R^{1/r}\gtrsim R^{2/r},\qquad|x|\le\frac{1}{100R}.
	\end{eqnarray}
	Define
	$$\Omega_R=\left\{x\in\mathbb{R}^2:|x|\le\frac{1}{100R}\right\}.$$
	Since
	$$|\Omega_R|=\pi c_1^2R^{-2}\approx R^{-2},$$
	the pointwise lower bound \eqref{eq--53} gives
 	\begin{eqnarray}\label{eq--54}
		\left\|V^r_{[0,1]}\{e^{it\Delta}f_R\}\right\|_{L^p(\mathbb{R}^2)}
		&\ge&\left\|V^r_{[0,1]}\{e^{it\Delta}f_R \}\right\|_{L^p(\Omega_R)} \gtrsim R^{2/r}|\Omega_R|^{1/p}\approx R^{2/r-2/p}.
	\end{eqnarray}

	It remains to calculate the Sobolev norm of $f_R$. By
	\eqref{eq--48}, the Fourier transform of $f_R$ is supported where
	$|\xi|\approx R$. Combining this observation with the normalization
	\eqref{eq--49}, we obtain
	\begin{eqnarray}
		\|f_R\|_{H^s(\mathbb{R}^2)}
		\approx
		R^s.
		\label{eq--55}
	\end{eqnarray}

	Applying the assumed variation estimate to $f_R$ and using
	\eqref{eq--54} and \eqref{eq--55}, we conclude that
	$$
	R^{2/r-2/p}
	\lesssim
	\left\|
	V^r_{[0,1]}\{e^{it\Delta}f_R \}
	\right\|_{L^p(\mathbb{R}^2)}
	\lesssim
	\|f_R\|_{H^s(\mathbb{R}^2)}
	\approx
	R^s.
	$$
	Since this inequality holds for arbitrarily large $R$, it is necessary
	that
	$$s\ge\frac{2}{r}-\frac{2}{p}.$$
	This proves the claimed necessary condition.
\end{proof}
\noindent\textbf{Data availability} This manuscript no associate data.

\noindent\textbf{A conflict of interest statement} This manuscript is original, has not been formally published,and is not under consideration for publication elsewhere. Both authors have approved this submission. The authors declare no competing interests.

\noindent\textbf{Statement on the use of AI.}
The authors developed the main ideas, wrote the proofs, and
determined the overall direction of this work.
We used GPT-6 Astra to help check the scalings and translations needed for the counterexample construction associated with Theorem \ref{th-5.4} and to refine the wording of the manuscript. Figures 1 and 2 were generated by GPT-6 Astra based on our descriptions of their intended meaning and scope. The proofs for Section 3 and Section 4 in this paper were completed by the authors.
We take full responsibility for the mathematical content of
the paper.


\begin{thebibliography}{99}

\bibitem{BBCR2011}
J. A. Barcel\'o, J. Bennett, A. Carbery, and K. M. Rogers,
\emph{On the dimension of divergence sets of dispersive equations},
Math. Ann. \textbf{349} (2011), no. 3, 599--622.
\newblock DOI: 10.1007/s00208-010-0529-z.

\bibitem{BORS2022}
D. Beltran, R. Oberlin, L. Roncal, A. Seeger, and B. Stovall,
\emph{Variation bounds for spherical averages},
Math. Ann. \textbf{382} (2022), nos. 1--2, 459--512.
\newblock DOI: 10.1007/s00208-021-02218-2.

\bibitem{BL1976}
J. Bergh and J. L\"ofstr\"om,
\emph{Interpolation Spaces: An Introduction},
Grundlehren der Mathematischen Wissenschaften, vol. 223,
Springer-Verlag, Berlin--New York, 1976.

\bibitem{B1988a}
J. Bourgain,
\emph{On the maximal ergodic theorem for certain subsets of the integers},
Israel J. Math. \textbf{61} (1988), no. 1, 39--72.
\newblock DOI: 10.1007/BF02776301.


\bibitem{B1992}
J. Bourgain,
\emph{A remark on Schr\"odinger operators},
Israel J. Math. \textbf{77} (1992), no. 1--2, 1--16.
\newblock DOI: 10.1007/BF02808007.

\bibitem{B1995}
J. Bourgain,
\emph{Some new estimates on oscillatory integrals},
in \emph{Essays on Fourier Analysis in Honor of Elias M. Stein},
Princeton Math. Ser., vol. 42,
Princeton University Press, Princeton, NJ, 1995, pp. 83--112.
\newblock DOI: 10.1515/9781400852949.83.

\bibitem{B2013}
J. Bourgain,
\emph{On the Schr\"odinger maximal function in higher dimension},
Proc. Steklov Inst. Math. \textbf{280} (2013), no. 1, 46--60.
\newblock DOI: 10.1134/S0081543813010045.

\bibitem{B1988b}
J. Bourgain,
\emph{On the pointwise ergodic theorem on $L^p$ for arithmetic sets},
Israel J. Math. \textbf{61} (1988), no. 1, 73--84.
\newblock DOI: 10.1007/BF02776302.



\bibitem{B1989}
J. Bourgain,
\emph{Pointwise ergodic theorems for arithmetic sets},
Publ. Math. Inst. Hautes \`Etudes Sci. \textbf{69} (1989), 5--45.
\newblock DOI: 10.1007/BF02698838.

\bibitem{B2016}
J. Bourgain,
\emph{A note on the Schr\"odinger maximal function},
J. Anal. Math. \textbf{130} (2016), no. 1, 393--396.
\newblock DOI: 10.1007/s11854-016-0042-8.

\bibitem{BD2015}
J. Bourgain and C. Demeter,
\emph{The proof of the $\ell^2$ decoupling conjecture},
Ann. of Math. (2) \textbf{182} (2015), no. 1, 351--389.
\newblock DOI: 10.4007/annals.2015.182.1.9.

\bibitem{BG2011}
J. Bourgain and L. Guth,
\emph{Bounds on oscillatory integral operators based on
	multilinear estimates},
Geom. Funct. Anal. \textbf{21} (2011), no.~6, 1239--1295.
\newblock DOI: 10.1007/s00039-011-0140-9.

\bibitem{BMSTW2018}
J. Bourgain, M. Mirek, E. M. Stein, and B. Wr\'obel,
\emph{On dimension-free variational inequalities for averaging operators
in $\mathbb{R}^d$},
Geom. Funct. Anal. \textbf{28} (2018), no. 1, 58--99.
\newblock DOI: 10.1007/s00039-018-0433-3.

\bibitem{CJRW2000}
J. T. Campbell, R. L. Jones, K. Reinhold, and M. Wierdl,
\emph{Oscillation and variation for the Hilbert transform},
Duke Math. J. \textbf{105} (2000), no. 1, 59--83.
\newblock DOI: 10.1215/S0012-7094-00-10513-3.

\bibitem{CJRW2003}
J. T. Campbell, R. L. Jones, K. Reinhold, and M. Wierdl,
\emph{Oscillation and variation for singular integrals
	in higher dimensions},
Trans. Amer. Math. Soc. \textbf{355} (2003), no. 5, 2115--2137.
\newblock DOI: 10.1090/S0002-9947-02-03189-6.

\bibitem{C1980}
L. Carleson,
\emph{Some analytic problems related to statistical mechanics},
in \emph{Euclidean Harmonic Analysis (Proc. Sem., Univ. Maryland,
College Park, Md., 1979)}, Lecture Notes in Math., vol. 779,
Springer, Berlin, 1980, pp. 5--45.
\newblock DOI: 10.1007/BFb0087666.

\bibitem{CDH2021}Y. Chen, Y. Ding, G. Hong and H. Liu, Variational inequalities for the commutators of rough
operators with BMO functions, Sci. China Math. 64 (2021), 2437–2460.

\bibitem{CDHL2018} Y. Chen, Y. Ding, G. Hong and H. Liu, Weighted jump and variational inequalities for rough
operators, J. Funct. Anal. 274 (2018), 2446–2475.

\bibitem{CHL2023}Y. Chen, G. Hong and J. Li, Quantitative weighted bounds for the q-variation of singular
integrals with rough kernels, J. Fourier Anal. Appl. 29 (2023), Paper No. 31, 50 pp.

\bibitem{DK1982}
B. E. J. Dahlberg and C. E. Kenig,
\emph{A note on the almost everywhere behavior of solutions to the
Schr\"odinger equation},
in \emph{Harmonic Analysis (Minneapolis, Minn., 1981)},
Lecture Notes in Math., vol. 908, Springer, Berlin--New York, 1982,
pp. 205--209.
\newblock DOI: 10.1007/BFb0093289.

\bibitem{DLTT2008}
C. Demeter, M. T. Lacey, T. Tao, and C. Thiele,
\emph{Breaking the duality in the return times theorem},
Duke Math. J. \textbf{143} (2008), no. 2, 281--355.
doi:10.1215/00127094-2008-020.

\bibitem{Du2017T}
X. Du,
\emph{A sharp Schr\"odinger maximal estimate in $\mathbb{R}^{2}$},
Ph.D. dissertation, University of Illinois at Urbana-Champaign, 2017.

\bibitem{DGL2017}
X. Du, L. Guth, and X. Li,
\emph{A sharp Schr\"odinger maximal estimate in $\mathbb{R}^2$},
Ann. of Math. (2) \textbf{186} (2017), no. 2, 607--640.
\newblock DOI: 10.4007/annals.2017.186.2.5.

\bibitem{DGLZ2018}
X. Du, L. Guth, X. Li, and R. Zhang,
\emph{Pointwise convergence of Schr\"odinger solutions and
	multilinear refined Strichartz estimates},
Forum Math. Sigma \textbf{6} (2018), e14, 18 pp.
\newblock DOI: 10.1017/fms.2018.11.

\bibitem{DL2019}
X. Du and X. Li,
\emph{$l^p$ decoupling for restricted $k$-broadness},
Math. Z. \textbf{292} (2019), no.~1--2, 725--737.
\newblock DOI: 10.1007/s00209-018-2178-y.

\bibitem{DKWZ2020}
X. Du, J. Kim, H. Wang, and R. Zhang,
\emph{Lower bounds for estimates of the Schr\"odinger maximal function},
Math. Res. Lett. \textbf{27} (2020), no. 3, 687--692.
\newblock DOI: 10.4310/MRL.2020.v27.n3.a4.

\bibitem{DZ2019}
X. Du and R. Zhang,
\emph{Sharp $L^2$ estimates of the Schr\"odinger maximal function in
higher dimensions},
Ann. of Math. (2) \textbf{189} (2019), no. 3, 837--861.
\newblock DOI: 10.4007/annals.2019.189.3.4.

\bibitem{GRY2020}
S. Guo, J. Roos, and P.-L. Yung,
\emph{Sharp variation-norm estimates for oscillatory integrals related
to Carleson's theorem},
Anal. PDE \textbf{13} (2020), no. 5, 1457--1500.
\newblock DOI: 10.2140/apde.2020.13.1457.

\bibitem{H1960}L. Hörmander, Estimates for translation invariant operators in $L^p$ spaces, Acta Math. 104 (1960), 93–140.

\bibitem{JSW2008}
R. L. Jones, A. Seeger, and J. Wright,
\emph{Strong variational and jump inequalities in harmonic analysis},
Trans. Amer. Math. Soc. \textbf{360} (2008), no. 12, 6711--6742.
\newblock DOI: 10.1090/S0002-9947-08-04538-8.

\bibitem{KT1998} M. Keel, T. Tao, Endpoint Strichartz estimates, Amer. J. Math. 120 (1998) 955–980.

\bibitem{K2019}
V. Kova\v{c},
\emph{Fourier restriction implies maximal and variational Fourier restriction},
J. Funct. Anal. \textbf{277} (2019), no. 10, 3355--3372.
\newblock DOI: 10.1016/j.jfa.2019.03.015.



\bibitem{L1976}
D. L\'epingle,
\emph{La variation d'ordre $p$ des semi-martingales},
Z. Wahrscheinlichkeitstheorie verw. Gebiete \textbf{36} (1976), no. 4,
295--316.
\newblock DOI: 10.1007/BF00532696.

\bibitem{L2006}
S. Lee,
\emph{On pointwise convergence of the solutions to Schr\"odinger
equations in $\mathbb{R}^2$},
Int. Math. Res. Not. \textbf{2006} (2006), Art. ID 32597, 21 pp.
\newblock DOI: 10.1155/IMRN/2006/32597.

\bibitem{LRS2012}
S. Lee, K. M. Rogers, and A. Seeger,
\emph{Improved bounds for Stein's square functions},
Proc. Lond. Math. Soc. (3) \textbf{104} (2012), no. 6, 1198--1234.
\newblock DOI: 10.1112/plms/pdr067.

\bibitem{LRS2013}
S. Lee, K. M. Rogers, and A. Seeger,
\emph{On space--time estimates for the Schr\"odinger operator},
J. Math. Pures Appl. (9) \textbf{99} (2013), no. 1, 62--85.
\newblock DOI: 10.1016/j.matpur.2012.06.003.

\bibitem{LRV2011}
S. Lee, K. M. Rogers, and A. Vargas,
\emph{An endpoint space--time estimate for the Schr\"odinger equation},
Adv. Math. \textbf{226} (2011), no. 5, 4266--4285.
\newblock DOI: 10.1016/j.aim.2010.12.006.

\bibitem{MT2014}
A. Mas and X. Tolsa,
\emph{Variation for the Riesz transform and uniform rectifiability},
J. Eur. Math. Soc. \textbf{16} (2014), no. 11, 2267--2321.
\newblock DOI: 10.4171/JEMS/487.

\bibitem{LMY2024}
J. Li, C. Miao, and A. Yu,
\emph{The space-time estimates for the Schr\"odinger equation},
arXiv:2402.13539 [math.CA], 2024.

\bibitem{L2023}
X. Li,
\emph{A Stein--Tomas type estimate and a decoupling inequality},
arXiv:2309.12835 (2023).
\newblock DOI: 10.48550/arXiv.2309.12835.

\bibitem{MVV1996}
A. Moyua, A. Vargas, and L. Vega,
\emph{Schr\"odinger maximal function and restriction properties of the Fourier transform},
International Mathematics Research Notices \textbf{1996} (1996), no. 16, 793--815.
doi:10.1155/S1073792896000499.



\bibitem{MST2017}
M. Mirek, E. M. Stein, and B. Trojan,
\emph{$\ell^p(\mathbb{Z}^d)$-estimates for discrete operators of Radon type:
variational estimates},
Invent. Math. \textbf{209} (2017), no. 3, 665--748.
\newblock DOI: 10.1007/s00222-017-0718-4.

\bibitem{MSZ2020a}
M. Mirek, E. M. Stein, and P. Zorin-Kranich,
\emph{Jump inequalities via real interpolation},
Math. Ann. \textbf{376} (2020), nos. 1--2, 797--819.
\newblock DOI: 10.1007/s00208-019-01889-2.



\bibitem{MSZ2020b}
M. Mirek, E. M. Stein, and P. Zorin-Kranich,
\emph{A bootstrapping approach to jump inequalities and their applications},
Anal. PDE \textbf{13} (2020), no. 2, 527--558.
\newblock DOI: 10.2140/apde.2020.13.527.

\bibitem{MSZ2020c}
M. Mirek, E. M. Stein, and P. Zorin-Kranich,
\emph{Jump inequalities for translation-invariant operators of Radon type on
$\mathbb{Z}^d$},
Adv. Math. \textbf{365} (2020), Art. 107065, 57 pp.
\newblock DOI: 10.1016/j.aim.2020.107065.

\bibitem{LR2017}
R. Luc\`a and K. M. Rogers,
\emph{Coherence on fractals versus pointwise convergence for the
Schr\"odinger equation},
Comm. Math. Phys. \textbf{351} (2017), no. 1, 341--359.
\newblock DOI: 10.1007/s00220-016-2722-8.

\bibitem{OSTTW2012}
R. Oberlin, A. Seeger, T. Tao, C. Thiele, and J. Wright,
\emph{A variation norm Carleson theorem},
J. Eur. Math. Soc. (JEMS) \textbf{14} (2012), no. 2, 421--464.
\newblock DOI: 10.4171/JEMS/307.

\bibitem{PX1988}
G. Pisier and Q. Xu,
\emph{The strong $p$-variation of martingales and orthogonal series},
Probab. Theory Related Fields \textbf{77} (1988), no. 4, 497--514.
\newblock DOI: 10.1007/BF00959613.

\bibitem{R2008}
K. M. Rogers,
\emph{A local smoothing estimate for the Schr\"odinger equation},
Adv. Math. \textbf{219} (2008), no. 6, 2105--2122.
\newblock DOI: 10.1016/j.aim.2008.08.008.

\bibitem{RS2010}
K. M. Rogers and A. Seeger,
\emph{Endpoint maximal and smoothing estimates for Schr\"odinger equations},
J. Reine Angew. Math. \textbf{640} (2010), 47--66.
\newblock DOI: 10.1515/crelle.2010.018.

\bibitem{RV2007}
K. M. Rogers and P. Villarroya,
Global estimates for the Schr\"odinger maximal operator,
Ann. Acad. Sci. Fenn. Math. \textbf{32} (2007), 425--435.

\bibitem{S1987}
P. Sj\"olin,
\emph{Regularity of solutions to the Schr\"odinger equation},
Duke Math. J. \textbf{55} (1987), no. 3, 699--715.
\newblock DOI: 10.1215/S0012-7094-87-05535-9.

\bibitem{SW2001}
E. M. Stein and S. Wainger,
\emph{Oscillatory integrals related to Carleson's theorem},
Math. Res. Lett. \textbf{8} (2001), nos. 5--6, 789--800.
\newblock DOI: 10.4310/MRL.2001.v8.n6.a9.


\bibitem{V1988}
L. Vega,
\emph{Schr\"odinger equations: pointwise convergence to the initial data},
Proc. Amer. Math. Soc. \textbf{102} (1988), no. 4, 874--878.
\newblock DOI: 10.1090/S0002-9939-1988-0934859-0.

\end{thebibliography}
\end{document}